\documentclass[11pt]{article}
\usepackage{amsmath,empheq} % \font used for R in Real numbers
\usepackage{verbatim}
\usepackage{amssymb, float}
\usepackage{color}
\usepackage{graphicx,epsfig}
\usepackage[applemac]{inputenc}
\title{Nonconvex integro-differential sweeping process with applications.}
\newtheorem{theorem}{Theorem}[section]

\newtheorem{lemma}[theorem]{Lemma}
\newtheorem{proposition}[theorem]{Proposition}
\newtheorem{definition}[theorem]{Definition}
\newtheorem{example}[theorem]{Example}

\def\beq{\begin{equation}}
\def\eeq{\end{equation}}
\def\baq{\begin{eqnarray}}
\def\eaq{\end{eqnarray}}
\def\baqn{\begin{eqnarray*}}
\def\eaqn{\end{eqnarray*}}

\def\image #1 (#2,#3) (echelle #4) #5{
\dimen2=#2
\dimen3=#3
\divide \dimen2 by 1000
\multiply \dimen2 by #4
\divide \dimen3 by 1000
\multiply \dimen3 by #4
\setbox1 =\vbox to \dimen2{\hsize=\dimen3\vfill\special{picture #1
scaled #4}}
\vbox{\hsize=\dimen3\box1\medskip\centerline{#5}}
}
\begin{document}
\author{Abderrahim Bouach\thanks{Laboratoire LMPEA,  Facult\'{e} des Sciences Exactes et Informatique,
Universit\'{e} Mohammed Seddik Benyahia, Jijel, B.P. 98, Jijel 18000, Alg\'{e}rie  ({\tt abderrahimbouach@gmail.com}).}
        \and Tahar Haddad\thanks{ Laboratoire LMPEA, Facult\'{e} des Sciences Exactes et Informatique,
Universit\'{e} Mohammed Seddik Benyahia, Jijel, B.P. 98, Jijel 18000, Alg\'{e}rie  ({\tt haddadtr2000@yahoo.fr}).}
           \and Lionel Thibault\thanks{ Universit\'{e} de Montpellier, Institut Montpelli\'{e}rain Alexander Grothendieck 34095 Montpellier CEDEX 5 France ({\tt lionel.thibault@umontpellier.fr}).}}

\maketitle
\begin{abstract}\noindent In this paper, we analyze and discuss the well-posedness of a new variant of the so-called sweeping process, introduced by J.J. Moreau in the early 70's \cite{More71}  with motivation in plasticity theory. In this variant, the normal cone to the (mildly
non-convex) prox-regular moving set $C(t)$, supposed to have an absolutely continuous variation, is perturbed by a sum of a Carath\'{e}odory mapping and an integral forcing term. The integrand of the forcing term depends on two time-variables, that is, we study a general integro-differential sweeping process of Volterra
type. By setting up an appropriate semi-discretization method combined with a new Gronwall-like inequality (differential inequality), we show that the integro-differential sweeping process has one and only one absolutely continuous solution. We also establish the continuity of the solution with respect to the initial value. The results of the paper are applied to the study of nonlinear integro-differential complementarity systems which are  combination of Volterra integro-differential equations with  nonlinear complementarity constraints. Another application is concerned with  non-regular electrical circuits containing time-varying capacitors and nonsmooth electronic device like diodes. Both applications represent an additional novelty of our paper.
\end{abstract}

\noindent {\bf Keywords}Moreau's sweeping process, Volterra integro-differential equation,  Differential complementarity systems,  Gronwall's inequality, Prox-regular sets, Differential inclusions.\\
\noindent {\bf AMS subject classifications} 49J40,47J20, 47J22,  34G25, 58E35, 74M15, 74M10, 74G25.
%\tableofcontents[hideallsubsections]
\setcounter{tocdepth}{1}
\tableofcontents

%%%%%%%%%%%%%%%%%%%%%%%%%%
\section{Introduction}
In the seventies, sweeping processes are introduced and deeply studied by J. J. Moreau
through a
series of papers, in particular \cite{More71,More77}. There, it is shown that such processes play an important role in elasto-plasticity, quasi-statics,
dynamics, especially in mechanics \cite{More71}. Roughly speaking, a point is swept by a moving closed
convex set $C(t)$ in a Hilbert space $H$, which can be formulated in the form of differential inclusion as
follows
\begin{equation}\label{eq0.1}
  \begin{cases} -\dot{x}(t)\in N_{C(t)}(x(t))\,\,\,\quad a.e.\,\,\,  t\in [T_{0}, T]&\\
   x(T_{0})=x_{0}\in C(T_{0}),
   \end{cases}
   \end{equation}
where $N_{C(t)}(\cdot)$ denotes the normal cone of $C(t)$ in the sense of convex analysis. When the systems are perturbed, it is natural to study the following variant
\begin{equation}\label{eq0.2}
  \begin{cases} -\dot{x}(t)\in N_{C(t)}(x(t))+f(t,x(t))\,\,\,\quad a.e.\,\,\,  t\in [T_{0}, T]&\\
   x(T_{0})=x_{0}\in C(T_{0})
   \end{cases}
   \end{equation}
where $f : [T_{0}, T] \times H \rightarrow H $ is a Carath\'{e}odory mapping.\\

Actually, diverse approaches for existence of solutions of (\ref{eq0.1}) and (\ref{eq0.2}) are available in the literature: Catching-up method (see, e.g., \cite{More77}), regularization procedure (see, e.g., \cite{More71,NacrThib1}), reduction to unconstrained differential inclusion (see, e.g., \cite{t}).
For the first and second methods, existence and uniqueness of solutions follow, to some extent, from the classical Gronwall inequality and the basic relation
\begin{equation}\label{eq0.3}
  \frac{d}{dt}\|x(t)\|^{2}=2\langle \dot x(t), x(t) \rangle,
   \end{equation}
(whenever meaningful), applied to two solutions or suitable approximate solutions $x_1,x_2$ of (\ref{eq0.2}), by means of the monotonicity of  $N_{C(t)}(\cdot)$ ( hypomonotonicity when $C(t)$ is prox-regular). Those features and the Lipschitz property of the forcing term $f$ with respect to the state variable are employed to obtain that the distance between $x_1(t)$ and $x_2(t)$ is nonincreasing with respect to time $t$. This reasoning allows in general the construction of a Cauchy sequence of approximate solutions, converging to a solution.

Several extensions of the sweeping process in diverse ways (well-posedness and optimal control) have been studied in
the literature (see, e.g., \cite{ah}, \cite{ha1}, \cite{bog}, \cite{tca}, \cite{KM}, \cite{t}, \cite{venl}  and references therein).

The present paper aims to study the following new variant of the sweeping process
\begin{equation}\label{eq1.1-N}
(P_{f_{1},f_{2}})  \begin{cases} -\dot{x}(t)\in N_{C(t)}(x(t))+f_{1}(t,x(t))+\displaystyle\int\limits_{T_{0}}^{t}f_{2}(t,s, x(s))ds\,\,\,\quad \text{a.e.}\,\,\,  t\in [T_{0}, T]&\\
   x(T_0)=x_{0}\in C(T_0),
   \end{cases}
   \end{equation}
where $ N_{C(t)}(\cdot)$ denotes the Clarke normal cone to the subset $C(t)$ of $H$.

%Let us set by $ P_{\Delta}:=\{(t,s)\in [T_{0},T]\times [T_{0},T]\, ;\,\,s\leq t \} $.\\

    We will have to use the following assumptions:
\begin{enumerate}
	\item [$ \mathcal{(}\mathcal{H}_{1}\mathcal{)} $]
       For each $ t\in [T_{0},T] $, $ C(t) $ is a nonempty closed subset of $ H $ which is r -prox-regular, for some $r\in (0,+\infty]$,
     and has an absolutely continuous
	variation,  in the sense that there is some absolutely continuous function $ \upsilon : [T_{0},T]\longrightarrow \mathbb{R} $ such that
\begin{equation*}
 C(t) \subset C(s)+ \lvert \upsilon(s)-\upsilon(t) \rvert\mathbb{B}_H, \quad  \forall \;s,t\in [T_{0},T],
\end{equation*}	
where $\mathbb{B}_H$ denotes the closed unit ball of $H$ centered at the origin.
\iffalse*******************
\begin{equation*}
	\lvert d_{C(s)}(y)-d_{C(t)}(y)\rvert\leq \lvert \upsilon(s)-\upsilon(t) \rvert, \quad  \forall \;t,s\in [T_{0},T],\,\,\forall y\in\,H  .
\end{equation*}
*******************\fi
	\item  [$ \mathcal{(}\mathcal{H}_{2}\mathcal{)} $]  $ f_{1} : [T_{0},T]\times H \longrightarrow H $ is (Lebesgue) measurable in time (i.e., $f(\cdot,x)$ is measurable for each $x\in H$), and  such that
\begin{enumerate}
	\item [$ \mathcal{(}\mathcal{H}_{2,1}\mathcal{)} $] there exist non-negative functions $ \beta_{1}(\cdot)\in L^{1}([T_{0},T],\mathbb{R}) $ such that
	\begin{equation*}
	\lVert f_{1}(t,x)\rVert\leq \beta_{1}(t)(1+\lVert x \rVert),\,\,\,\text{for all}\,\,\,t\in[T_{0},T]\,\,\text{and for any}\,\,\,x\in\underset{t\in [T_{0},T]}\bigcup C(t){\color{blue};}
	\end{equation*}
	\item [$ \mathcal{(}\mathcal{H}_{2,2}\mathcal{)} $] for each real $ \eta>0 $ there exists a non-negative function $ L_{1}^{\eta}(\cdot)\in L^{1}([T_{0},T],\mathbb{R}) $  such that for any $ t\in [T_{0},T] $ and for any $ (x,y)\in {B}[0,\eta]\times {B}[0,\eta] $,
	\begin{equation*}
	\lVert f_{1}(t,x)-f_{1}(t,y) \rVert\leq L_{1}^{\eta}(t)\lVert x-y \rVert,
	\end{equation*}	
where $B[0,\eta]$ denotes the closed ball centered at the origin with radius $\eta$.
\end{enumerate}	
\item  [$ \mathcal{(}\mathcal{H}_{3}\mathcal{)} $]  $ f_{2} : Q_{\Delta}\times H \longrightarrow H $ is a measurable mapping such that
\begin{enumerate}
	\item [$ \mathcal{(}\mathcal{H}_{3,1}\mathcal{)} $] there exists a non-negative function $ \beta_{2}(\cdot,\cdot)\in L^{1}(Q_{\Delta},\mathbb{R}) $ such that
	\begin{equation*}
	\lVert f_{2}(t,s,x)\rVert\leq \beta_{2}(t,s)(1+\lVert x \rVert),\,\,\,\text{for all}\,\,\,(t,s)\in Q_{\Delta}\,\,\text{and for any}\,\,\,x\in\underset{t\in [T_{0},T]}\bigcup C(t){\color{blue};}
	\end{equation*}
	\item [$ \mathcal{(}\mathcal{H}_{3,2}\mathcal{)} $] for each real $ \eta>0 $ there exists a non-negative function
	$ L_{2}^{\eta}(\cdot)\in L^{1}([T_{0},T],\mathbb{R}) $  such that for
all $ (t,s)\in Q_{\Delta} $ and for any $ (x,y)\in {B}[0,\eta]\times {B}[0,\eta]$,
	\begin{equation*}
	\lVert f_{2}(t,s,x)-f_{2}(t,s,y) \rVert\leq L_{2}^{\eta}(t)\lVert x-y \rVert.
	\end{equation*}	
\end{enumerate}	
\end{enumerate}
Above $L^1([T_0,T],\mathbb{R})$ (resp. $L^1(Q_{\Delta},\mathbb{R})$) stands for the space of Lebesgue
integrable functions on $[T_0,T]$ (resp. $Q_{\Delta}$), where
$$
   Q_{\Delta}:=\{(t,s)\in [T_0,T]\times[T_0,T]: s \leq t\}.
$$

We called the differential inclusion (\ref{eq1.1-N}) as integro-differential sweeping process because the integral of the state and the velocity are defined in the dynamical system. One can interpret (\ref{eq1.1-N}) as follows: as long as $x(t)$ is in the interior of the set $C(t)$, we get
$N_{C(t)}(x(t)) = 0$ and (\ref{eq1.1-N}) reduces to a Volterra integro-differential equation
\begin{equation}\label{eq1.2}
  \begin{cases} -\dot{x}(t)=f_{1}(t,x(t))+\displaystyle\int\limits_{T_{0}}^{t}f_{2}(t,s, x(s))ds\,\,\,\quad a.e.\,\,\,  t\in [T_0, T]&\\
   x(T_0)=x_{0}
   \end{cases}
   \end{equation}
 (for at least a small period of time) to satisfy the constraint $x(t)\in C(t)$, until $x(t)$ hits the boundary of the set $C(t)$. At this moment, if the vector field $-(f_{1}(t,x(t))+\displaystyle\int\limits_{T_{0}}^{t}f_{2}(t,s, x(s))ds)$ is pointed outside of the set $C(t)$, then any component of this vector field in the direction normal to $C(t)$ at $x(t)$ must be annihilated to maintain the motion of $x$ within the constraint set. So, the system (\ref{eq1.1-N}) can be considered as a Volterra integro-differential equation (\ref{eq1.2}) under control term $u(t) \in N_{C(t)}(\cdot)$ which guarantees that the trajectory $x(t)$ always belongs to the desired set $C(t)$ for all $t \in [T_0, T]$.\\

The well-posedness of the classical perturbed sweeping process (\ref{eq0.2}), i.e.,$P_{f_{1},0}$ ( $f_2 \equiv 0 $), has been studied by many authors with different assumptions on data, see, e.g., \cite{edm,edmt,nacthi} and references therein. Sweeping process involving integral perturbation i.e., $P_{0,f_{2}}$ ( $f_1 \equiv 0 $)
 was considered earlier by  Brenier, Gangbo and Savare \cite{bgs} and recently by Colombo and Kozaily \cite{CK}. In the
latter paper \cite{CK} the authors proved the existence and uniqueness of solution
 with the particular integral $\displaystyle\int\limits_{0}^{t}f_{2}(s, x(s))ds$, i.e., for the following problem
 \begin{equation*}\label{eq1.1bis}
(P_{0,f_{2}})  \begin{cases} -\dot{x}(t)\in N_{C(t)}(x(t))+\displaystyle\int\limits_{0}^{t}f_{2}(s, x(s))ds\,\,\,\quad \text{a.e.}\,\,\,  t\in [0, T]&\\
   x(0)=x_{0}\in C(0).
   \end{cases}
   \end{equation*}

It is also worth mentioning that Colombo and Kozaily say in their
paper \cite{CK}: "\emph{of course, existence and uniqueness to $(P_{0,f_{2}})$ is not surprising}". Indeed, we observe that with the above integral $\displaystyle\int\limits_{0}^{t}f_{2}(s, x(s))ds$, the integro-differential sweeping process $(P_{f_{1},f_{2}})$ is equivalent to
\begin{equation*}
-\dot{x}(t)\in N_{C(t)}(x(t))+f_{1}(t,x(t))+y(t),\,\dot{y}(t)=f_{2}(t,x(t)),\,\,\,x(0)=x_{0},\,y(0)=0,
\end{equation*}
and so
\begin{equation*}
\overbrace{
\begin{pmatrix}
-\dot{x}(t)\\
-\dot{y}(t)
\end{pmatrix}}^{-\dot{X}(t)}\in  N_{C(t)\times H}\overbrace{
\begin{pmatrix}
x(t)\\
y(t)
\end{pmatrix}}^{X(t)}+\overbrace{
\begin{pmatrix}
f_{1}(t,x(t))+y(t)\\
-f_{2}(t,x(t))
\end{pmatrix}}^{f(t,X(t))},
\end{equation*}
which is a special case of the classical perturbed sweeping process (\ref{eq0.2}); see, e.g., \cite{More77,nacthi} for the situation of unbounded moving sets.\\
So, in \cite{CK}, the aim of the authors was not the well-posedness. Their motivation for studying $(P_{0,f_{2}})$ was designing a smoother method of
penalization, the motivation of which comes from applications to deriving necessary optimality conditions for optimal control problems with sweeping processes.\\

Now, if the integral involving $f_{2}$ depends on two \emph{time}-variables as in \eqref{eq1.1-N},  the reduction of $(P_{f_{1},f_{2}})$ there to the perturbed sweeping process \eqref{eq0.2} cannot be applied.\\

To the best of our knowledge, for the problem under consideration in the case of the
function $f_{2}$ depending on two \emph{time}-variables, that is, in the case of a general integro-differential sweeping process of Volterra
type $(P_{f_{1},f_{2}})$, a well-posedness result, including the existence, uniqueness, and
stability of the solution, has not been obtained up to the present time.\\

In the present paper, we obtain results on the existence and uniqueness of a
solution to the Volterra sweeping process  $(P_{f_{1},f_{2}})$ in
a Hilbert space.  This is done with the help of a new Gronwall-like inequality (see Section \ref{s3}) and of a new scheme corresponding to the existence of absolutely continuous solutions for the quasi-autonomous sweeping processes
\begin{equation*}
\left\{
\begin{array}{l}
-\dot{x}_{n}(t)\in N_{C(t)}(x_{n}(t))+f_{1}(t,x_{n}(t_{k}))+\displaystyle\sum\limits_{j=0}^{k-1}\displaystyle\int\limits_{t_{j}}^{t_{j+1}}f_{2}(t,s,x_{n}(t_{j}))\,ds\\
+\displaystyle\int\limits_{t_{k}}^{t}f_{2}(t,s,x_{n}(t_{k}))\,ds\;\; \text{ a.e. }t\in \left[ t_{k},t_{k+1}\right], \\
{x}_{n}(T_0)=x_0 \in C(T_0),
\end{array}%
\right.
\end{equation*}
where $T_0= t_0 < t_1 < ... < t_n = T$ is a discretization of the interval $[T_0, T]$.\\

 The outline of  the paper is as follows. In Section \ref{s2}, we recall some preliminary results that we use throughout. In Section \ref{s3}, we prove a new Gronwall-like inequality ( differential inequality). Then, in Section \ref{s4}, we present our main existence,
uniqueness, and stability result. In Section \ref{s5} we use those  results in the study of nonlinear integro-differential complementarity systems. This is realized by transforming such systems into integro-differential sweeping processes of the form \eqref{eq1.1-N} where the moving set $C(t)$ is described by a finite number of inequalities. We also provide sufficient verifiable conditions ensuring the absolute continuity of the moving set. Finally, in Section \ref{s6}, we give a second application of our results to non-regular electrical circuits containing time-varying capacitors and nonsmooth electronic device like diodes. Both applications represent an additional novelty of our paper.

\section{Notation and preliminaries}

Throughout $H$ is a real Hilbert space endowed with the inner
product $\langle \cdot,\cdot\rangle$ and associated norm $\Vert
\cdot\Vert$. As usual, we will denote by $\mathbb{B}_H$ or $\mathbb{B}$ the closed
unit ball of $H$ and by $B(x,\delta)$ (resp. $B[x,\delta]$) the open (resp.
closed) ball around $x\in H$ with radius $\delta>0$. For a nonempty
subset $S$ of $H$ the associated distance function is denoted by
$d_S$, that is,
\begin{equation*}
d_{S}(x):=\underset{y\in S}{\inf }\left\Vert x-y\right\Vert \text{
for all } x\in H.
\end{equation*}
By $\mathcal{C}([T_0,T],H)$ we denote the space of
continuous mappings from $[T_0,T]$ into $H$ equipped with the
supremum norm $\Vert \cdot \Vert_{\infty}$, where we recall that $-\infty<T_0 < T<+\infty$. As usual $\mathbb{R}$ will denote the set of real numbers, $\mathbb{R}_+$ the set of non-negative reals, that is, $\mathbb{R}_+:=[0,+\infty)$,  and $\mathbb{N}$ the set positive integers. In various cases, it will be convenient to use the notation write $I:=[T_0,T]$

The Clarke tangent cone of  $S$ at  $x\in S$, denoted by $T^{C}(S;x)$, is the set of $h\in H$ such that, for every sequence $(x_{n})_{n\in \mathbb{N}}$ of  $S$ with $x_{n} \longrightarrow x$ as $n\longrightarrow\infty$ and for every sequence $(t_{n})_{n\in \mathbb{N}}$ of positive reals with $t_{n}\longrightarrow 0$ as $n\longrightarrow\infty$, there exists a sequence $(h_{n})_{n\in \mathbb{N}}$ of $H$ with $h_{n} \longrightarrow h$ as $n\longrightarrow\infty$ satisfying $x_{n}+t_{n}h_{n}\in S$ for all $n\in \mathbb{N}$. This set is obviously a cone containing zero and it is known to be closed and convex. The polar cone of $T^{C}(S;x)$  is the Clarke normal cone $N^{C}(S;x)$ of $S$ at $x$, that is
\begin{equation*}
N^{C}_{S}(x):=\{\upsilon\in H\,:\,\,\langle \upsilon , h \rangle \leq 0,\,\,\forall\,h\in T^{C}(S;x)\} .
\end{equation*}
If  $x\notin S$, by convention $T^{C}(S;x)$ and $N^{C(S;x)}$ are empty.

For a function $f:H\to\mathbb{R} \cup\{+ \infty\}$, The Clarke subdifferential $\partial _{C}f(x)$ of $f$ at $x$ ( see \cite{R.W1}  )  with $f(x)<+\infty$ is defined by
\begin{equation*}
\partial _{C}f(x):=\{ \upsilon\in H\,:\,\,(\upsilon,-1)\in N^{C}(epi\,f ; (x,f(x))) \} {\color{red},}
\end{equation*}
where $H\times \mathbb{R}$ is endowed with the usual product structure and epi $f$ is the epigraph of $f$ , that is,
\begin{equation*}
epi\,f:=\{ (x,r)\in H\times \mathbb{R}\,:\,\,f(x)\leq r \} .
\end{equation*}
If $f$ is not finite at $x$, we see that $\partial_{C}f(x)=\emptyset$. In addition to the latter definition, there is another link
between the Clarke normal cone and the Clarke subdifferential, given by $\partial_{C}\psi_S(x)=N^{C}(S;x)$, where  $\psi_S$ denotes the indicator function of the subset $S$ of $H$, i.e. $\psi(y)=0$ if $y\in S$ and $\psi(y)= + \infty$ if $y\notin S$ .

A vector $v\in H$ is a {\it proximal subgradient} of $f$ at a point $x$ with $f(x)<+\infty$ (see, e.g., \cite{C.L.S.W,Mord,R.W}) if there exist some reals $\sigma \geq 0$ and $\delta >0$ such that
\begin{equation*}
  \langle v,y-x\rangle \leq f(y)-f(x)+\sigma \|y-x\|^2 \quad \text{for all } y \in B(x,\delta).
\end{equation*}
The set $\partial_{P}f(x)$ of all proximal subgradients of $f$ at
$x$ is the {\it proximal subdifferential} of $f$ at $x$. Of course,
$\partial_{P}f(x)=\emptyset$ if $f(x)=+\infty$. Note that one always has $\partial_{P}f(x)\subset \partial_{C}f(x)$ .\\
Taking the proximal subdifferential $\partial_{P}\psi_S(x)$ of the
indicator function $\psi_S$  we obtain the {\it proximal
normal cone} $N_{S}^{P} (x)$ of $S$ at $x$. So, a vector $v\in H$ is a {\it
proximal normal vector} of $S$ at $x\in S$ if and only if there are
reals $\sigma\geq 0$ and $\delta>0$ such that
\begin{equation}\label{eq2.1}
  \langle v,y-x \rangle \leq \sigma\|y-x \|^2 \quad \text{for all } y\in S\cap B(x,\delta).
\end{equation}
So, we always have the inclusion $N^{P}_{S}(x)\subset N^{C}_{S}(x)$ for all $x\in S$ .

   The proximal normal cone can be described in the following
geometrical way (see, e.g., \cite{C.L.S.W})
\begin{equation}\label{eq2.2}
N^{P}_{S}(x)=\{v \in H:\;\exists r>0\text{ such that }x\in
\mathrm{Proj}_{S}(x+r v )\},
\end{equation}
where
\begin{equation*}
\mathrm{Proj}_{S}(u):=\{y\in S:\; d_{S}(u) =\left\Vert
u-y\right\Vert \}.
\end{equation*}
The proximal normal cone is also connected with the distance
function to $S$ through the equalities (see, e.g., \cite{C.L.S.W})
\begin{equation*}
\partial_{P}d_{S}(x)=N_{S}^{P}(x)\cap \mathbb{B}_{H} \text{ and }   N^P_{S}(x)=\mathbb{R}_{+}\partial_{P}d_{S}(x),
\end{equation*}
where $\mathbb{R}_{+} :=[0, +\infty[$.

   In many cases one has to require in (\ref{eq2.2}) that the constant
$r$ be uniform for all the unit proximal normal vectors of $S$.
The sets which satisfy that property are known as (uniformly)
prox-regular sets. Given $r \in ]0,+\infty ]$ the closed subset
$S$ is {\it (uniformly)} $r$-{\it prox-regular} (see
\cite{P.R.T}) (called also $r$-positively reached (see
\cite{Fede}) or $r$-{\it proximally smooth} (see \cite{C.S.W})
provided that, for every $x\in S$ and every unit vector $v\in
N^{P}_{S}(x)$ one has $ x\in \mathrm{Proj}_S (x+ r v).$
The latter is equivalent to
\begin{equation*}
  S \cap B_{H}(x+r v, r)= \emptyset \text{ or equivalently }   \left\langle v ,x'-x \right\rangle
  \leq \frac{1}{2r }\left\Vert x'-x\right\Vert ^{2},
\end{equation*}
for all $x'\in S$. Of course, in the latter inequality, $\frac{1}{r
}=0$ for $r =+\infty$ (as usual). It is worth pointing out that
for $r=+\infty$, the uniform $r$-prox-regularity of the closed
set $S$ amounts to its convexity.
\begin{definition}\label{def}
For a given $ r\in(0,\infty] $, a subset $ S $ of the Hilbert space $H$ is uniformly r-prox-regular, or  $r$-prox-regular for short,
if and only if for all $ x\in S $
 and all  $ 0 \neq\varsigma\in N_{S}^{P}(x) $ one has
\begin{equation*}
\langle \dfrac{\varsigma}{\lVert \varsigma \rVert} , y-x \rangle\leq \dfrac{\lVert y-x \rVert^{2}}{2r},\,\,\,\forall\,y\in S .
\end{equation*}
\end{definition}

  The following propositions summarize some important consequences of prox-regularity needed in the paper. For the
proof of these results, we refer the reader to \cite{P.R.T,ct}.
\begin{proposition}
Let $ S $ be a nonempty closed set in $ H $  which is uniformly r-prox-regular for some $r\in [0,+\infty]$. Then  for any $ x_{i}\in S $, $ \varsigma_{i}\in N^{P}_{S}(x_{i}) $  with $ i=1,2 $ one has :
\begin{equation*}
\langle \varsigma_{2}-\varsigma_{1} , x_{2}-x_{1} \rangle\geq -\dfrac{1}{2}\bigg(\dfrac{\lVert \varsigma_{2} \rVert + \lVert \varsigma_{1} \rVert }{r}\bigg)\lVert x_{2}-x_{1} \rVert^{2} .
\end{equation*}
\end{proposition}

\begin{proposition}\label{proj}
Let $ S $ be a nonempty closed subset in $ H $ and let $ r\in(0,\infty] $.
If the subset $ S $ is uniformly r-prox-regular then the following hold:
\begin{enumerate}
	\item[(a)]  The proximal and Clarke normal cones of $S$ coincide. \\
	\item[(b)] for all $ x\in H $ with $ d_{S}(x)< r $, $ \mathrm{Proj}_{S}(x) $ is nonempty and is a singleton set .\\
	\item[(c)] the Clarke and the proximal subdifferentials of $ d_{S} $ coincide at all points $ x\in H $ with
	$ d_{S}(x)< r $ .	
\end{enumerate}
\end{proposition}

   The assertion (a) in Proposition \ref{proj} leads us to put  $
N_{S}(x):=N_{S}^{C}(x)=N_{S}^{P}(x) $ whenever the set $S$ is $r$-prox-regular.

\begin{proposition}\label{dist}
Let $ S $ be a nonempty closed subset in $ H $ and let $ r\in(0,\infty] $.
If the subset $ S $ is uniformly r-prox-regular then the following hold:
\begin{enumerate}
	\item[(a)]   For any $ x\in S $ and any $ \varsigma\in\partial^{P}d_{S}(x) $ one has  for any $ y\in H $ such that $ d_{S}(y)< r $
	\begin{equation*}
\langle \varsigma , y-x \rangle	\leq \dfrac{2}{r}\lVert y-x \rVert^{2} + d_{S}(y) .
	\end{equation*}
	\item[(b)] For any $ x\in H $ with 	$ d_{S}(x)< r $, the proximal subdifferential $ \partial^{P}d_{S}(x) $ is a
	nonempty closed convex subset in $ H $.		
\end{enumerate}	
\end{proposition}
\section{Gronwall-like differential inequality}\label{s3}\ \\

We start this section with the following continuous Gronwall's inequality \cite{Sh}.

\begin{lemma}[Gronwall's inequality]\label{m}
	Let  $ T>T_0 $ be given reals and  $ a(\cdot) , b(\cdot) \in L^{1}([T_0,T];\mathbb{R}) $ with  $ b(t) \geq 0  $ for almost all $ t\in [T_0,T] $.  Let the absolutely continuous function $ w : [T_0,T] \longrightarrow \mathbb{R}_{+} $ satisfy
	\begin{equation*}
	(1-\alpha)w^\prime(t)\leq a(t)w(t) + b(t)w^{\alpha}(t) , \hspace{0.3cm} a.e.\,\, t\in [T_0,T],
	\end{equation*}
	where $ 0\leq \alpha < 1 $. Then for all $ t\in [T_0,T] $, one has
	\begin{equation*}
	w^{1-\alpha}(t)\leq w^{1-\alpha}(T_0)\exp(\int_{T_0}^{t} a(\tau)d\tau) + \int_{T_0}^{t}\exp(\int_{s}^{t} a(\tau)d\tau)b(s)ds .
	\end{equation*}
\end{lemma}
We will need the following lemma which is a straightforward consequence of Gronwall's
lemma.
\begin{lemma}\label{22}
	Let $ \rho : [T_{0},T]\longrightarrow \mathbb{R} $  be a nonnegative absolutely continuous function and let $ b_{1}, b_{2}, a : [T_{0},T]\longrightarrow \mathbb{R}_{+} $ be non-negative Lebesgue integrable functions. Assume that
	\begin{equation}\label{23}
	\dot{\rho}(t)\leq a(t)+b_{1}(t)\rho(t)+b_{2}(t)\int\limits_{T_{0}}^{t}\rho(s)\,ds, \hspace{0.3cm} a.e.\,\, t\in [T_{0},T].
	\end{equation}	
	Then for all $ t\in [T_{0},T] $, one has
	\begin{align*}
	 \rho(t)\leq  \rho(T_{0})\,\exp\bigg(\int\limits_{T_{0}}^{t}(b(\tau)+1)\,d\tau\bigg) + \int\limits_{T_{0}}^{t}a(s)\,\exp\bigg(\int\limits_{s}^{t}(b(\tau)+1)\,d\tau\bigg)\,ds,
	\end{align*}	
	where $ b(t):=\max\{b_{1}(t),b_{2}(t)\} $, a.e. $ t\in[T_{0},T] $  .
\end{lemma}
\begin{proof}
Put $ b(t)=\max\{b_{1}(t),b_{2}(t)\} $, a.e. $ t\in[T_{0},T] $. Setting $z(t)=\int\limits_{T_{0}}^{t}\rho(s)\,ds\Rightarrow\dot{z}(t)=\rho(t)$, $\ddot{z}(t)=\dot{\rho}(t)$ . Then from \eqref{23}
we see that
\begin{align*}
\ddot{z}(t)&\leq a(t)+b_{1}(t)\dot{z}(t)+b_{2}(t)z(t)\leq  a(t)+b(t)w( t),
\end{align*}
where  $w(t)=\dot{z}(t)+z(t)$,  for all $ t\in[T_{0},T] $. Hence, for a.e. $ t\in[T_{0},T] $
$$
   \dot{w}(t)=\ddot{z}(t)+\dot{z}(t) \quad\text{and}\quad
   \dot{w}(t)\leq a(t)+(b(t)+1)w(t) .
$$
Applying the Gronwall Lemma \ref{m} with $ w$, one obtains for all $ t\in [T_{0},T] $
	\begin{equation*}
	w(t) \leq w(T_{0})\exp\big(\int\limits_{T_{0}}^{t}(b(\tau)+1)\,d\tau\big)
	+\int\limits_{T_{0}}^{t}a(s)\,\exp\big(\int\limits_{s}^{t}(b(\tau)+1)\,d\tau\big)\,ds,
	\end{equation*}
which gives
\begin{align*}
	 \rho(t)\leq\dot{z}(t)+z(t)=w(t)\leq  \rho(T_{0})\,\exp\bigg(\int\limits_{T_{0}}^{t}(b(\tau)+1)\,d\tau\bigg) + \int\limits_{T_{0}}^{t}a(s)\,\exp\bigg(\int\limits_{s}^{t}(b(\tau)+1)\,d\tau\bigg)\,ds .
	\end{align*}	
\end{proof}

We establish now the following new Gronwall-like lemma.

\begin{lemma}[Gronwall-like differential inequality]\label{8}
	Let $ \rho : [T_{0},T]\longrightarrow \mathbb{R} $  be a non-negative absolutely continuous function  and let $ K_{1},K_{2},\varepsilon : [T_{0},T]\longrightarrow \mathbb{R}_{+} $ be non-negative Lebesgue integrable functions. Suppose for some $ \epsilon> 0 $
	\begin{equation}\label{5}
	\dot{\rho}(t)\leq \varepsilon(t)+\epsilon+ K_{1}(t)\rho(t)+K_{2}(t)\sqrt{\rho(t)}\int\limits_{T_{0}}^{t}\sqrt{\rho(s)}\,ds, \hspace{0.3cm} a.e.\,\, t\in [T_{0},T].
	\end{equation}	
	Then for all $ t\in[T_{0},T] $, one has
	\begin{align*}
	 \sqrt{\rho(t)}&\leq  \sqrt{\rho(T_{0})+\epsilon}\,\exp\bigg(\int\limits_{T_{0}}^{t}(K(s)+1)\,ds\bigg) + \dfrac{\sqrt{\epsilon}}{2}\int\limits_{T_{0}}^{t}\exp\bigg(\int\limits_{s}^{t}(K(\tau)+1)\,d\tau\bigg)\,ds  \\
	&+2\bigg(\sqrt{\int\limits_{T_{0}}^{t}\varepsilon(s)\,ds+\epsilon}
	-\sqrt{\epsilon}\,\exp\bigg(\int\limits_{T_{0}}^{t}(K(\tau)+1)\,d\tau\bigg)\bigg)\\
	& + 2\int\limits_{T_{0}}^{t}(K(s)+1)\exp\bigg(\int\limits_{s}^{t}(K(\tau)+1)\,d\tau\bigg)
	\sqrt{\int\limits_{T_{0}}^{s}\varepsilon(\tau)\,d{\tau}+\epsilon}\,\,ds,
	\end{align*}	
	where $ K(t):=\max\bigg\{\dfrac{K_{1}(t)}{2},\dfrac{K_{2}(t)}{2}\bigg\} $, a.e. $ t\in[T_{0},T] $  .
\end{lemma}
\begin{proof}
	Set $ \lambda(t)=\sqrt{\int\limits_{T_{0}}^{t}\varepsilon(s)\,ds+\epsilon} $ and $ z_{\varepsilon}(t)=\sqrt{\rho(t)+\lambda^{2}(t)}$ for all $ t\in[T_{0},T] $ .\\
	From \eqref{5} we have for a.e. $ t\in[T_{0},T] $
	\begin{equation}\label{6}
	\dot{\rho}(t)\leq \varepsilon(t)+\epsilon+ K_{1}(t)(\rho(t)+\lambda^{2}(t))+K_{2}(t)\sqrt{\rho(t)+\lambda^{2}(t)}\int\limits_{T_{0}}^{t}\sqrt{\rho(s)+\lambda^{2}(s)}\,ds
	\end{equation}
and
$$
  \dot{z}_{\varepsilon}(t)=\dfrac{\dot{\rho}(t)+2\dot{\lambda}(t)\lambda(t)}{2\sqrt{\rho(t)+\lambda^{2}(t)}}=\dfrac{\dot{\rho}(t)+\varepsilon(t)}{2z_{\varepsilon}(t)},
\quad\text{or equivalently}\;
	\dot{\rho}(t)=2z_{\varepsilon}(t)\dot{z}_{\varepsilon}(t)-\varepsilon(t),
$$
hence from \eqref{6}
	\begin{equation*}
	2z_{\varepsilon}(t)\dot{z}_{\varepsilon}(t)\leq 2\varepsilon(t)+\epsilon+K_{1}(t)z_{\varepsilon}(t)^{2}+K_{2}(t)z_{\varepsilon}(t)\int\limits_{T_{0}}^{t}z_{\varepsilon}(s)\,ds.
	\end{equation*}
	Therefore,
	\begin{align*}
	\dot{z}_{\varepsilon}(t)  \leq \dfrac{\varepsilon(t)}{z_{\varepsilon}(t)}+\dfrac{\epsilon}{2z_{\varepsilon}(t)}+\dfrac{K_{1}(t)}{2}z_{\varepsilon}(t)+\dfrac{K_{2}(t)}{2}\int\limits_{T_{0}}^{t}z_{\varepsilon}(s)\,ds \leq 2\dot{\lambda}(t)+\dfrac{\sqrt{\epsilon}}{2}+\dfrac{1}{2}\left(K_{1}(t)z_{\varepsilon}(t) +K_{2}(t)\int\limits_{T_{0}}^{t}z_{\varepsilon}(s)\,ds\right).
	\end{align*}
	Since
	\begin{gather*}
	\lambda(t)=\sqrt{\int\limits_{T_{0}}^{t}\varepsilon(s)\,ds+\epsilon}\leq \sqrt{\rho(t)+\int\limits_{T_{0}}^{t}\varepsilon(s)\,ds+\epsilon}=\sqrt{\rho(t)+\lambda^{2}(t)}=z_{\varepsilon}(t),
	\end{gather*}
	then
	\begin{align*}
	\dfrac{1}{z_{\varepsilon}(t)}&\leq \dfrac{1}{\lambda(t)}, \quad \text{or equivalently}\;
	\dfrac{\varepsilon(t)}{z_{\varepsilon}(t)}\leq\dfrac{\varepsilon(t)}{\lambda(t)}.
	\end{align*}
	Also we have  $ \dot{\lambda}(t)=\dfrac{\varepsilon(t)}{2\lambda(t)} $. Then $\dfrac{\varepsilon(t)}{z_{\varepsilon}(t)}\leq 2\dot{\lambda}(t)$, and
	 $\sqrt{\epsilon}\leq \sqrt{\epsilon+\int\limits_{T_{0}}^{t}\varepsilon(s)\,ds}=\lambda(t)\leq z_{\varepsilon}(t)$, hence $\dfrac{\epsilon}{2z_{\varepsilon}(t)}\leq\dfrac{\sqrt{\epsilon}}{2}$ . Letting  $ K(t):=\max\bigg\{\dfrac{K_{1}(t)}{2},\dfrac{K_{2}(t)}{2}\bigg\} $ and applying
the	Gronwall Lemma \ref{22} with $ z_{\varepsilon} $, one obtains for all $ t\in [T_{0} , T ] $
	\begin{align*}
	z_{\varepsilon}(t)&\leq z_{\varepsilon}(T_{0})\exp\big(\int\limits_{T_{0}}^{t}(K(\tau)+1)\,d\tau\big)+\int\limits_{T_{0}}^{t}\exp\big(\int\limits_{s}^{t}(K(\tau)+1)\,d\tau\big)(2\dot{\lambda}(s)+\dfrac{\sqrt{\epsilon}}{2})\,ds\\
	 &=\sqrt{\rho(T_{0})+\epsilon}\exp(\int\limits_{T_{0}}^{t}(K(s)+1)\,ds)+\int\limits_{T_{0}}^{t}\exp\big(\int\limits_{s}^{t}(K(\tau)+1)\,d\tau\big)(2\dot{\lambda}(s)+\dfrac{\sqrt{\epsilon}}{2} )\,ds
\end{align*}
or equivalently
\begin{align*}
 z_{\varepsilon}(t)
	 \leq \sqrt{\rho(T_{0})+\epsilon}\exp\big(\int\limits_{T_{0}}^{t}(K(s)+1)\,ds\big)+2\int\limits_{T_{0}}^{t}\exp\big(\int\limits_{s}^{t}(K(\tau)+1)\,d\tau\big)\dot{\lambda}(s)\,ds +\dfrac{\sqrt{\epsilon}}{2}\int\limits_{T_{0}}^{t}\exp\big(\int\limits_{s}^{t}(K(\tau)+1)\,d\tau\big)\,ds .
	\end{align*}
On the other hand, from integration by parts, we note that
\begin{align*}
  \int\limits_{T_{0}}^{t}\exp\big(\int\limits_{s}^{t}(K(\tau)+1)\,d\tau\big)\dot{\lambda}(s)\,ds &=[\exp\big(\int\limits_{s}^{t}(K(\tau)+1)\,d\tau\big)\lambda(s)]_{T_{0}}^{t}
 +\int\limits_{T_{0}}^{t}(K(s)+1)\exp\big(\int\limits_{s}^{t}(K(\tau)+1)\,d\tau\big)\lambda(s)\,ds\\
 &=\lambda(t)-\exp\big(\int\limits_{T_{0}}^{t}(K(\tau)+1)\,d\tau\big)\sqrt{\epsilon}
  +\int\limits_{T_{0}}^{t}(K(s)+1)\exp\big(\int\limits_{s}^{t}(K(\tau)+1)\,d\tau\big)\lambda(s)\,ds,
\end{align*}
which combined with what precedes gives
\begin{align*}
	z_{\varepsilon}(t)&\leq  \sqrt{\rho(T_{0})+\epsilon}\,\exp\big(\int\limits_{T_{0}}^{t}(K(s)+1)\,ds\big) + \dfrac{\sqrt{\epsilon}}{2}\int\limits_{T_{0}}^{t}\exp\big(\int\limits_{s}^{t}(K(\tau)+1)\,d\tau\big)\,ds \\
	&+2\lambda(t)-2\exp\big(\int\limits_{T_{0}}^{t}(K(\tau)+1)\,d\tau\big)\sqrt{\epsilon} + 2\int\limits_{T_{0}}^{t}(K(s)+1)\exp\big(\int\limits_{s}^{t}(K(\tau)+1)\,d\tau\big)\lambda(s)\,ds .
\end{align*}
Consequently, observing that $\sqrt{\rho(t)}\leq \sqrt{\rho(t)+\lambda^{2}(t)}=z_{\varepsilon}(t)$ we obtain  	
\begin{align*}
 \sqrt{\rho(t)} &\leq
\sqrt{\rho(T_{0})+\epsilon}\,\exp\big(\int\limits_{T_{0}}^{t}(K(s)+1)\,ds\big) + \dfrac{\sqrt{\epsilon}}{2}\int\limits_{T_{0}}^{t}\exp\big(\int\limits_{s}^{t}(K(\tau)+1)\,d\tau\big)\,ds  +2\lambda(t)\\
\\&-2\exp\big(\int\limits_{T_{0}}^{t}(K(\tau)+1)\,d\tau\big)\sqrt{\epsilon} + 2\int\limits_{T_{0}}^{t}(K(s)+1)\exp\big(\int\limits_{s}^{t}(K(\tau)+1)\,d\tau\big)\lambda(s)\,ds,
\end{align*}	
which completes the proof of the lemma .
\end{proof}
%%%%%%%%%%%%%%%%%%%%%%%%%%%%%%%%%%%%%%%%%%%%%%%%%%%%%%%%%%%%%%%%%%%%%%%%%%%%%%%%%%%%%%%%%%%%%%%%%%%%%%%%%%%%%%%%%%%%%%%%%%%%%%%%%%%%%%%%%%%%%%%%%%%%%%%%%%%%%%%%%%%%%%%%%%%%%%%%%%%%%%%%%%%%%%%%%%%%%%%%%%%%%%%%%%
%%%%%%%%%%%%%%%%%%%%%%%%%%%%%%%%%%%%%%%%%%%%%%%%%%%%%%%%%%%%%%%%%%%%%%%%%%%%%%%%%%%%%%%%%%%%%%%%%%%%%%%%%%%%%%%%%%%%%%%%%%%%%%%%%%%%%%%%%%%%%%%%%%%%%%%%%%%%%%%%%%%%%%%%%%%%%%%%%%%%%%%%%%%%%%%%%%%%%%%%%%%%%%%%%%%%%%%%%%%%%%%%%%%%%%%%%%%%%%%%%%%%%%%%%%%%%%%%%%%%%%%%%%%%%%%%%%%%%%%%
\section{Main results}\label{s4}\ \\

In this section, we give and prove our main results in the study of the integro-differential sweeping process $(P_{f_{1},f_{2}})$. They concern the existence,
uniqueness, and continuous dependence of the solution with respect to the initial data.
We state first in the next proposition a result which will be utilized in our development.  Clearly, when the sets $C(t)$ are bounded, the hypothesis $(\mathcal{H}_1)$ is ensured by the usual Hausdorff variation hypothesis
$$
  \mathrm{haus}\big(C(s),C(t)\big) \leq \lvert \upsilon(s)-\upsilon(t) \rvert, \quad  \forall \;t,s\in [T_{0},T],
$$
which according to the equality
$\mathrm{haus}(S,S')=\mathop{\sup}\limits_{y\in H}|d_S(y)-d_{S'}(y)|$
for bounded sets $S,S'$, amounts to requiring for the bounded sets $C(s), C(t)$ that
\begin{equation}\label{eq1-VarDist}
	\lvert d_{C(s)}(y)-d_{C(t)}(y)\rvert\leq \lvert \upsilon(s)-\upsilon(t) \rvert, \quad  \forall \;t,s\in [T_{0},T],\,\,\forall y\in\,H  .
\end{equation}
The result is proved in \cite{cst,edm} under the hypothesis \eqref{eq1-VarDist} but the proof in
\cite{cst,edm} is valid with the hypothesis $(\mathcal{H}_1)$.
One of the advantages of $(\mathcal{H}_1)$ is that the sets $C(t)$ there need not be bounded.

\begin{proposition}\label{prop1.1}
Let $H$ be a real Hilbert space, suppose that $C(\cdot)$ satisfies $\mathcal{(}\mathcal{H}_{1}\mathcal{)}$. Let $h:\left[ T_{0},T\right] \longrightarrow H$ be a single-valued mapping in $L^{1}(\left[ T_{0},T\right] ,H)$. Then for any $ x_{0}\in C(T_{0}) $ there exists a unique absolutely continuous solution $x(\cdot)$ for the following differential inclusion
\begin{equation*}
\left\{
\begin{array}{l}
-\dot{x}(t)\in N_{C(t)}(x(t)) + h(t)\text{ a.e \ }t\in \left[ T_{0},T\right],\\
x(T_{0})=x_{0}.
\end{array}%
\right.
\end{equation*}%
Moreover $x(\cdot)$ satisfies the following inequality
\begin{equation}\label{ine_1_prop1.1}
\left\Vert \dot{x}(t)+h(t)\right\Vert \leq\left\Vert h(t)\right\Vert + \vert \dot{\upsilon}(t) \vert\;\; \text{a.e. }t\in \left[ T_{0},T\right].
\end{equation}
\end{proposition}
\begin{theorem}\label{exist}
Let $H$ be a  real Hilbert space and assume that $ \mathcal{(}\mathcal{H}_{1}\mathcal{)} $, $ \mathcal{(}\mathcal{H}_{2}\mathcal{)} $ and $ \mathcal{(}\mathcal{H}_{3}\mathcal{)} $ are satisfied. Then for any initial point $x_{0} \in H,$ with $x_{0}\in C(T_0)$ there exists a unique absolutely continuous solution $x : [T_0,T]\longrightarrow H $ of the differential inclusion $(P_{f_{1},f_{2}})$. This solution satisfies:
\begin{enumerate}
\item For\,\, a.e.\,\,$t\in [T_0,T]$
 \begin{equation}\label{30}
\lVert \dot{x}(t)+f_{1}(t,x(t))+\int\limits_{T_0}^{t}f_{2}(t,s,x(s))\,ds \rVert\leq \lvert\dot{\upsilon}(t)\rvert+\lVert f_{1}(t,x(t))\rVert+\int\limits_{T_0}^{t}\lVert f_{2}(t,s,x(s))\rVert\,ds .
\end{equation}
 \item If $\displaystyle\int\limits_{T_0}^T\bigg[\beta_{1}(\tau)+\int\limits_{T_0}^{\tau}\beta_{2}(\tau,s)\,ds\bigg] d\tau<\dfrac{1}{4}$, one has
\begin{equation}\label{31}
\left\Vert f_{1}(t,x(t))\right\Vert \leq (1+M)\beta_{1}(t),\,\,\,\text{for all}\,\,\,t\in[T_0,T],
\end{equation}
\begin{equation}\label{32}
\left\Vert f_{2}(t,s,x(s))\right\Vert\leq (1+M)\beta_{2}(t,s),\,\,\,\text{for all}\,\,\,(t,s)\in Q_{\Delta},
\end{equation}
and for almost all $t\in [T_0,T]$
\begin{equation}\label{33}
\left\Vert \dot{x}(t)+f_{1}(t,x(t))+\displaystyle\int\limits_{T_{0}}^{t}f_{2}(t,s,x(s))\,ds\right\Vert\leq (1+M)\bigg(\beta_{1}(t)+\int\limits_{T_{0}}^{t}\beta_{2}(t,s)\,ds\bigg)+\vert \dot{\upsilon}(t) \vert,
\end{equation}
where $ M:=2\bigg(\left\Vert x_{0} \right\Vert+\displaystyle\int\limits_{T_{0}}^{T} \vert \dot{\upsilon}(\tau) \vert\,d\tau+\dfrac{1}{2}\bigg) $.
\item Assume  the following strengthened form of
assumption $ \mathcal{(}\mathcal{H}_{3,1}\mathcal{)} $ on the function $f_{2}$ holds:\\
$ \mathcal{(}\mathcal{H}^{\prime}_{3,1}\mathcal{)} $ : there exist  non-negative  functions $ \alpha(\cdot)\in L^{1}([T_0,T],\mathbb{R}) $ and\\  $ g(\cdot)\in L^{1}(P_{\Delta},\mathbb{R}) $ such that
\begin{equation*}
\lVert f_{2}(t,s,x) \rVert\leq  g(t,s) +\alpha(t)\lVert x \rVert,\,\,\,\text{for any}\,\, (t,s)\in Q_{\Delta}\,\,\,\text{and any}\,\,\,x\in\underset{t\in [T_0,T]}\bigcup C(t).
\end{equation*}
  Then we have
\begin{gather}
%\lVert x(t) \rVert \leq \tilde{l},\label{40}\\
\lVert f_{1}(t,x(t)) \rVert\leq (1+\widetilde{M})\beta_{1}(t),\,\,\,\text{for all}\,\,\,t\in[T_0,T],\label{41}\\
\lVert f_{2}(t,s,x(s)) \rVert \leq g(t,s) + \alpha(t)\widetilde{M},\,\,\,\text{a.e.}\,\,\,(t,s)\in Q_{\Delta},\label{42}
\end{gather}
and for almost all $t \in [T_0,T]$
\begin{gather}
\lVert \dot{x}(t)+f_{1}(t,x(t))+\int\limits_{T_0}^{t}f_{2}(t,s,x(s))\,ds \rVert\leq \rvert\dot{\upsilon}(t)\rvert+ (1+\widetilde{M})\beta_{1}(t)+\int\limits_{T_0}^{t}g(t,s)\,ds+T\alpha(t)\widetilde{M},\label{43}
\end{gather}
where
\begin{align*}
 \widetilde{M}:= \lVert x_{0} \rVert\exp\bigg(\int\limits_{T_0}^{T}(b(\tau)+1)\,d\tau\bigg)+\exp\bigg(\int\limits_{T_0}^{T}(b(\tau)+1)\,d\tau\bigg)\int\limits_{T_0}^{T}\bigg(\rvert\dot{\upsilon}(s)\rvert +2\beta_{1}(s)+2\int\limits_{T_0}^{T}g(s,\tau)\,d\tau\bigg)\,ds,
\end{align*}
\begin{equation*}
\text{and}\,\,\,b(t):=2\max\{\beta_{1}(t),\alpha(t)\}\,\,\, \text{for all}\,\,\, t\in [T_0,T].
\end{equation*}
\end{enumerate}
\end{theorem}
\begin{proof}
The proof of existence of solution is divided in several steps.\\
\textbf{Step 1.} \textbf{Discretization of the interval }$I=\left[T_{0},T\right].$  \newline
For each $n\in\mathbb{N}$,  divide the interval $I$ into $n$ intervals of
length $h=\frac{T-T_{0}}{n}$ and  define, for all $i\in \left\{
0,\cdots,n-1\right\} $
\begin{equation}
\left\{
\begin{array}{l}
t_{i+1}^{n}:=t_{i}^{n}+h=T_{0}+ih, \\
I_{i}^{n}:=\left[ t_{i}^{n},t_{i+1}^{n}\right],
\end{array}%
\right.
\end{equation}%
so that
\begin{equation}
T_{0}=t_{0}^{n}<t_{1}^{n}<\cdots<t_{i}^{n}<t_{i+1}^{n}<\cdots<t_{n}^{n}=T.
\end{equation}
\textbf{Step 2. Construction of the sequence  $x_{n}(\cdot)$.} \newline
We construct a sequence of mappings $(x_{n}(\cdot))_{n\in\mathbb{N}}$ in $\mathcal{C}
(I,H)$ which converges uniformly to a solution $x(\cdot)$ of $(P)$.\\
Our method consists in establishing a  sequence of discrete solutions\; $(x_{k}^{n}(\cdot))_{n\in\mathbb{N}}$\; in each interval $I_{k}^{n}:=\left[
t_{k}^{n},t_{k+1}^{n}\right]  (0\leq k\leq n-1$) by using
Proposition \ref{prop1.1}. Indeed, we proceed as follows.
\newline
Given the following problem
\begin{equation}
(P_{0}):\left\{
\begin{array}{l}
-\dot{x}(t)\in N_{C(t)}(x(t))+f_{1}(t,x_{0})+\displaystyle\int\limits_{T_{0}}^{t}f_{2}(t,s,x_{0})\,ds\;\;\text{ a.e. \ }t\in \left[
T_{0},t_{1}^{n}\right],  \\
x(T_{0})=x_{0}.
\end{array}%
\right.
\end{equation}%
Then $(P_0)$ is a perturbed  sweeping process with the perturbation depending only on time.\\
Let $h_0:[T_0,t^n_1]\to H$ be defined by  $ h_{0}(t):=f_{1}(t,x_{0})+\displaystyle\int\limits_{T_{0}}^{t}f_{2}(t,s,x_{0})\,ds$ for all $t\in \left[
T_{0},t_{1}^{n}\right] $. We see by integrable linear growth condition that
 \begin{equation*}
 \int\limits_{T_0}^{T}\Vert h_{0}(t) \Vert dt \leq (1+\Vert x_{0} \Vert)\int\limits_{T_0}^{T} \beta_{1}(t)dt+(1+\Vert x_{0} \Vert)\int\limits_{T_0}^{T}\int\limits_{T_0}^{t} \beta_{2}(t,s)ds\,dt,
 \end{equation*}
and since $ \beta_{1}(\cdot)\in L^{1}([T_{0},T],\mathbb{R}_{+}) $ and $ \beta_{2}(\cdot)\in L^{1}(Q_{\Delta},\mathbb{R}_{+}) $, then $ h_{0}(\cdot) $ is an integrable function. Therefore, by Proposition \ref{prop1.1} the differential inclusion $(P_0)$ has a unique absolutely continuous solution denoted by
\begin{equation}
x_{0}^{n}(\cdot):[ T_{0},t_{1}^{n}] \longrightarrow H,
\end{equation}
satisfying the following inequality
\begin{equation}
\left\Vert \dot{x}_{0}^{n}(t)+f_{1}(t,x_{0})+\displaystyle\int\limits_{T_{0}}^{t}f_{2}(t,s,x_{0})\,ds\right\Vert \leq \left\Vert f_{1}(t,x_{0})+\displaystyle\int\limits_{T_{0}}^{t}f_{2}(t,s,x_{0})\,ds \right\Vert+\vert \dot{\upsilon}(t) \vert\;
\end{equation}
$ \text{ a.e. }t\in \left[ T_{0},t_{1}^{n}\right]. $\\
Next, let us consider the following problem
\begin{equation}
(P_{1}):\left\{
\begin{array}{l}
-\dot{x}(t)\in N_{C(t)}(x(t))+f_{1}(t,x_{0}^{n}(t_{1}^{n}))+\displaystyle\int\limits_{T_{0}}^{t^{n}_{1}}f_{2}(t,s,x_{0})\,ds\displaystyle\int\limits_{t^{n}_{1}}^{t}f_{2}(t,s,x^{n}_{0}(t^{n}_{1}))\,ds\;\text{ a.e. }t\in
\left[ t_{1}^{n},t_{2}^{n}\right], \\
x(t_{1}^{n})=x_{0}^{n}(t_{1}^{n}).
\end{array}%
\right.
\end{equation}%
Let $h_1:[t^n_1,t^n_2] \to H$ be defined by
$$
 h_{1}(t):=f_{1}(t,x_{0}^{n}(t_{1}^{n}))+\displaystyle\int\limits_{T_{0}}^{t^{n}_{1}}f_{2}(t,s,x_{0})\,ds+\displaystyle\int\limits_{t_{1}^{n}}^{t}f_{2}(t,s,x_{0}^{n}(t_{1}^{n}))\,ds\;\text{for all}\; t\in \left[
t_{1}^{n},t_{2}^{n}\right].
$$
 We can see by integrable linear growth conditions that
 \begin{align*}
 \int\limits_{T_0}^{T}\Vert h_{1}(t) \Vert dt &\leq (1+\Vert x_{0}^{n}(t_{1}^{n}) \Vert)\int\limits_{T_0}^{T} \beta_{1}(t)dt+(1+\Vert x_{0} \Vert)\int\limits_{T_0}^{T}\int\limits_{T_0}^{t^{n}_{1}} \beta_{2}(t,s)ds\,dt+(1+\Vert x_{0}^{n}(t_{1}^{n}) \Vert)\int\limits_{T_0}^{T}\int\limits_{t^{n}_{1}}^{t} \beta_{2}(t,s)ds\,dt\\
 & \leq (1+\text{max}\{\Vert x_{0}^{n}(t_{1}^{n}) \Vert,\Vert x_{0}\Vert\})\bigg(\int\limits_{T_0}^{T} \beta_{1}(t)dt+\int\limits_{T_0}^{T}\int\limits_{T_0}^{t} \beta_{2}(t,s)ds\,dt\bigg).
 \end{align*}
 We know from the above problem $ (P_{0}) $ that the mapping $ x_{0}^{n}(\cdot) $ is absolutely continuous, then in particular bounded on $ [T_0,T] $. Further,
 since $ \beta_{1}(\cdot)\in L^{1}([T_{0},T],\mathbb{R}_{+}) $ and $ \beta_{2}(\cdot)\in L^{1}(Q_{\Delta},\mathbb{R}_{+}) $, then $ h_{1}(\cdot) $ is an integrable mapping. The same arguments as above show that $(P_{1})$ has a unique  absolutely continuous solution denoted by
\begin{equation}
x_{1}^{n}(\cdot):\left[ t_{1}^{n},t_{2}^{n}\right] \longrightarrow H,
\end{equation}
and this solution satisfies the following inequality
\begin{align}
&\left\Vert \dot{x}_{1}^{n}(t)+f_{1}(t,x_{0}^{n}(t_{1}^{n}))+\displaystyle\int\limits_{T_{0}}^{t^{n}_{1}}f_{2}(t,s,x_{0})\,ds+\displaystyle\int\limits_{t^{n}_{1}}^{t}f_{2}(t,s,x^{n}_{0}(t^{n}_{1}))\,ds\right\Vert\notag\\
& \leq \left\Vert f_{1}(t,x_{0}^{n}(t_{1}^{n}))+\displaystyle\int\limits_{T_{0}}^{t^{n}_{1}}f_{2}(t,s,x_{0})\,ds+\displaystyle\int\limits_{t^{n}_{1}}^{t}f_{2}(t,s,x^{n}_{0}(t^{n}_{1}))\,ds \right\Vert+\vert \dot{\upsilon}(t) \vert
\;\;\text{ a.e. }t\in \left[ t_{1}^{n},t_{2}^{n}\right].
\end{align}
Successively, for each $n$, we have a finite sequence of absolutely continuous mappings $(x_{k}^{n}(\cdot))_{0\leq k\leq n-1}$ with for each $k\in \{0,\cdots,n-1\}$
\begin{equation}
x_{k}^{n}(\cdot):\left[ t_{k}^{n},t_{k+1}^{n}\right] \longrightarrow H
\end{equation}%
 such that
\begin{equation}
(P_{k-1}):\left\{
\begin{array}{l}
-\dot{x}_{k}^{n}(t)\in N_{C(t)}(x_{k}^{n}(t))+f_{1}(t,x_{k-1}^{n}(t_{k}^{n}))+\displaystyle\sum\limits_{j=0}^{k-1}\displaystyle\int\limits_{t^{n}_{j}}^{t^{n}_{j+1}}f_{2}(t,s,x^{n}_{j-1}(t^{n}_{j}))\,ds\\
+\displaystyle\int\limits_{t^{n}_{k}}^{t}f_{2}(t,s,x^{n}_{k-1}(t^{n}_{k}))\,ds\;\; \text{ a.e. }t\in \left[ t_{k}^{n},t_{k+1}^{n}\right]. \\
x_{k}^{n}(t_{k}^{n})=x_{k-1}^{n}(t_{k}^{n}),%
\end{array}%
\right.\label{s11}
\end{equation}
where for $k=0$ we put $x_{-1}^{n}(T_{0}):=x_{0}$.
Moreover
\begin{align}
&\left\Vert \dot{x}_{k}^{n}(t)+f_{1}(t,x_{k-1}^{n}(t_{k}^{n}))+\displaystyle\sum\limits_{j=0}^{k-1}\displaystyle\int\limits_{t^{n}_{j}}^{t^{n}_{j+1}}f_{2}(t,s,x^{n}_{j-1}(t^{n}_{j}))\,ds+\displaystyle\int\limits_{t^{n}_{k}}^{t}f_{2}(t,s,x^{n}_{k-1}(t^{n}_{k}))\,ds,\right\Vert\notag\\
& \leq \left\Vert f_{1}(t,x_{k-1}^{n}(t_{k}^{n}))+\displaystyle\sum\limits_{j=0}^{k-1}\displaystyle\int\limits_{t^{n}_{j}}^{t^{n}_{j+1}}f_{2}(t,s,x^{n}_{j-1}(t^{n}_{j}))\,ds+\displaystyle\int\limits_{t^{n}_{k}}^{t}f_{2}(t,s,x^{n}_{k-1}(t^{n}_{k}))\,ds, \right\Vert+\vert \dot{\upsilon}(t) \vert,   \label{s12}
\end{align}
$ \text{ a.e. }t\in \left[ t_{k}^{n},t_{k+1}^{n}\right]. $\\
Defining for each $k\in\{0,1,\cdots,n-1\}$ the mapping $h_k:[t^n_k,t^n_{k+1}]\to H$ by
\begin{equation*}
h_{k}(t):=f_{1}(t,x_{k-1}^{n}(t_{k}^{n}))+\displaystyle\sum\limits_{j=0}^{k-1}\displaystyle\int\limits_{t^{n}_{j}}^{t^{n}_{j+1}}f_{2}(t,s,x^{n}_{j-1}(t^{n}_{j}))\,ds+\displaystyle\int\limits_{t^{n}_{k}}^{t}f_{2}(t,s,x^{n}_{k-1}(t^{n}_{k}))\,ds,
\end{equation*}
for all $t\in \left[t_{k}^{n},t_{k+1}^{n}\right] $.
 Clearly, by integrable linear growth conditions we have
 \begin{align*}
 \int\limits_{T_0}^{T}\Vert h_{k}(t) \Vert dt &\leq (1+\Vert x_{k-1}^{n}(t_{k}^{n}) \Vert)\int\limits_{T_0}^{T} \beta_{1}(t)dt+\displaystyle\sum\limits_{j=0}^{k-1}(1+\Vert x_{j-1}^{n}(t_{j}^{n}) \Vert)\int\limits_{t^{n}_{j}}^{t^{n}_{j+1}} \beta_{2}(t,s)ds\,dt\\
 &  + (1+\Vert x_{k-1}^{n}(t_{k}^{n}) \Vert)\int\limits_{T_0}^{T}\int\limits_{t^{n}_{k}}^{t}\beta_{2}(t,s)ds\,dt\\
 & \leq (1+\underset{0\leq j\leq k-1}{\max }\left\Vert x_{j-1}^{n}(t_{j}^{n})\right\Vert )\bigg(\int\limits_{T_0}^{T} \beta_{1}(t)dt+\int\limits_{T_0}^{T}\int\limits_{T_0}^{t} \beta_{2}(t,s)ds\,dt\bigg).
 \end{align*}
 We know from the above problems $ (P_{j})_{0\leq j \leq k-1} $ that the mapping $ x_{k-1}^{n}(\cdot) $ is absolutely continuous, then in particular bounded on $ [T_0,T] $. Further,
 since $ \beta_{1}(\cdot)\in L^{1}([T_{0},T],\mathbb{R}_{+}) $ and $ \beta_{2}(\cdot)\in L^{1}(P_{\Delta},\mathbb{R}_{+}) $; the mapping $ h_{k}(\cdot) $ is integrable
on $[t^n_k,t^n_{k+1}]$.

   Now, we define the sequence $(x_n(\cdot))_n$ from the discrete sequences $(x_{k}^{n}(.))$ as follows.\\
For each $n\in\mathbb{N}$, let $x_{n}(\cdot):\left[ T_{0},T\right]\longrightarrow H$ such that
\begin{equation}
x_{n}(t):=x_{k}^{n}(t),\text{ if }t\in \left[ t_{k}^{n},t_{k+1}^{n}\right].  \label{s13}
\end{equation}
It is obvious from this definition that $x_{n}(\cdot)$ is absolutely continuous.  \newline
Let $\theta _{n}(\cdot):\left[ T_{0},T\right]\longrightarrow \left[ T_{0},T\right]$ be defined by
\begin{equation}
\left\{
\begin{array}{l}
\theta _{n}(T_{0}):=T_{0}, \\
\theta _{n}(t):=t_{k}^{n}\text{, if }t\in \left] t_{k}^{n},t_{k+1}^{n}\right].
\end{array}
\right.   \label{s14}
\end{equation}%
We obtain from  \eqref{s11}, \eqref{s12}, \eqref{s13}, \eqref{s14}, that
\begin{equation}
\left\{
\begin{array}{l}
-\dot{x}_{n}(t)\in N_{C(t)}(x_{n}(t))+f_{1}(t,x_{n}(\theta _{n}(t)))+\displaystyle\int\limits_{T_{0}}^{t}f_{2}(t,s,x_{n}(\theta _{n}(s)))\,ds\text{ a.e.
 }t\in \left[ T_{0},T\right],  \\
x_{n}(T_{0})=x_{0},%
\end{array}%
\right.   \label{swdpte}
\end{equation}%
and a.e. $ t\in \left[ T_{0},T\right] $ we have
\begin{align}
&\left\Vert \dot{x}_{n}(t)+f_{1}(t,x_{n}(\theta _{n}(t)))+\displaystyle\int\limits_{T_{0}}^{t}f_{2}(t,s,x_{n}(\theta _{n}(s)))\,ds\right\Vert\notag\\
& \leq \left\Vert f_{1}(t,x_{n}(\theta _{n}(t)))+\displaystyle\int\limits_{T_{0}}^{t}f_{2}(t,s,x_{n}(\theta _{n}(s)))\,ds \right\Vert+\vert \dot{\upsilon}(t) \vert.\label{es0}
\end{align}
\textbf{Step 3. We show that the sequence $(\dot{x}_{n}(\cdot))$ is uniformly dominated by an integrable function.}\newline
Since $\beta_{1}(\cdot)\in L^1([T_{0},T],\mathbb{R}_+)$ and $ \beta_{2}(\cdot,\cdot)\in L^1(P_{\Delta},\mathbb{R}_+) $ we suppose without loss of generality that
\begin{equation}\label{less_beta}
\int\limits_{T_0}^T\bigg[\beta_{1}(\tau)+\int\limits_{T_0}^{\tau}\beta_{2}(\tau,s)\,ds\bigg] d\tau<\dfrac{1}{4}.
\end{equation}
By construction we have for each $i\in \left\{ 0,\cdots,n-1\right\} $ and for
a.e. $t\in \left[ t_{i}^{n},t_{i+1}^{n}\right]$
\begin{align*}
&\left\Vert \dot{x}_{n}(t)+f_{1}(t,x_{n}(t_{i}^{n}))+\displaystyle\sum\limits_{j=0}^{i-1}\displaystyle\int\limits_{t^{n}_{j}}^{t^{n}_{j+1}}f_{2}(t,s,x_{n}(t^{n}_{j}))\,ds+\displaystyle\int\limits_{t^{n}_{i}}^{t}f_{2}(t,s,x_{n}(t^{n}_{i}))\,ds\right\Vert\\
& \leq \left\Vert f_{1}(t,x_{n}(t_{i}^{n}))+\displaystyle\sum\limits_{j=0}^{i-1}\displaystyle\int\limits_{t^{n}_{j}}^{t^{n}_{j+1}}f_{2}(t,s,x_{n}(t^{n}_{j}))\,ds+\displaystyle\int\limits_{t^{n}_{i}}^{t}f_{2}(t,s,x_{n}(t^{n}_{i}))\,ds \right\Vert + \vert \dot{\upsilon}(t) \vert.
\end{align*}
According to $ \mathcal{(}\mathcal{H}_{2,1}\mathcal{)} $ and $ \mathcal{(}\mathcal{H}_{3,1}\mathcal{)} $ we have
\begin{align*}
\left\Vert \dot{x}_{n}(t)\right\Vert&\leq   2\left\Vert f_{1}(t,x_{n}(t_{i}^{n}))\right\Vert+2\displaystyle\sum\limits_{j=0}^{i-1}\displaystyle\int\limits_{t^{n}_{j}}^{t^{n}_{j+1}}\left\Vert f_{2}(t,s,x_{n}(t^{n}_{j})) \right\Vert\,ds+2\displaystyle\int\limits_{t^{n}_{i}}^{t}\left\Vert f_{2}(t,s,x_{n}(t^{n}_{i})) \right\Vert\,ds+\vert \dot{\upsilon}(t) \vert\\
&\leq 2(1+\underset{0\leq k\leq n}{\max }\left\Vert x_{n}(t_{k}^{n})\right\Vert )\beta_{1}(t)+2(1+\underset{0\leq k\leq n}{\max }\left\Vert
x_{n}(t_{k}^{n})\right\Vert )\displaystyle\sum\limits_{j=0}^{i-1}\displaystyle\int\limits_{t^{n}_{j}}^{t^{n}_{j+1}}\beta_{2}(t,s)\,ds\\
&+2(1+\underset{0\leq k\leq n}{\max }\left\Vert x_{n}(t_{k}^{n})\right\Vert )\displaystyle\int\limits_{t^{n}_{i}}^{t}\beta_{2}(t,s)\,ds+\vert \dot{\upsilon}(t) \vert\\
&=\vert \dot{\upsilon}(t) \vert + 2(1+\underset{0\leq k\leq n}{\max }\left\Vert x_{n}(t_{k}^{n})\right\Vert )\beta_{1}(t)+2(1+\underset{0\leq k\leq n}{\max }\left\Vert x_{n}(t_{k}^{n})\right\Vert )\displaystyle\int\limits_{T_{0}}^{t}\beta_{2}(t,s)\,ds,
\end{align*}
and then
\begin{align*}
\left\Vert x_{n}(t^{n}_{i+1}) \right\Vert&\leq \left\Vert x_{n}(t^{n}_{i}) \right\Vert +\int\limits_{t^{n}_{i}}^{t^{n}_{i+1}} \vert \dot{\upsilon}(\tau) \vert\,d\tau + 2(1+\underset{0\leq k\leq n}{\max }\left\Vert x_{n}(t_{k}^{n})\right\Vert )\int\limits_{t^{n}_{i}}^{t^{n}_{i+1}}\beta_{1}(\tau)\,d\tau\\
&+2(1+\underset{0\leq k\leq n}{\max }\left\Vert x_{n}(t_{k}^{n})\right\Vert )\int\limits_{t^{n}_{i}}^{t^{n}_{i+1}}\int\limits_{T_{0}}^{\tau}\beta_{2}(\tau,s)\,ds\,d\tau .
\end{align*}
Iterating, it follows that
\begin{align*}
\left\Vert x_{n}(t^{n}_{i+1}) \right\Vert&\leq \left\Vert x_{0} \right\Vert +\sum\limits_{k=0}^{i}\int\limits_{t^{n}_{k}}^{t^{n}_{k+1}} \vert \dot{\upsilon}(\tau) \vert\,d\tau
 + 2(1+\underset{0\leq j\leq n}{\max }\left\Vert x_{n}(t_{j}^{n})\right\Vert )\sum\limits_{k=0}^{i}\int\limits_{t^{n}_{k}}^{t^{n}_{k+1}}\beta_{1}(\tau)\,d\tau\\
& +2(1+\underset{0\leq j\leq n}{\max }\left\Vert x_{n}(t_{j}^{n})\right\Vert )\sum\limits_{k=0}^{i}\int\limits_{t^{n}_{k}}^{t^{n}_{k+1}}\int\limits_{T_{0}}^{\tau}\beta_{2}(\tau,s)\,ds\,d\tau
\end{align*}
This yields  the following inequality
\begin{align}
\left\Vert x_{n}(t^{n}_{i+1}) \right\Vert&\leq \left\Vert x_{0} \right\Vert +\int\limits_{T_{0}}^{t^{n}_{i+1}} \vert \dot{\upsilon}(\tau) \vert\,d\tau + 2(1+\underset{0\leq k\leq n}{\max }\left\Vert x_{n}(t_{k}^{n})\right\Vert )\int\limits_{T_{0}}^{t^{n}_{i+1}}\beta_{1}(\tau)\,d\tau\notag\\
&+2(1+\underset{0\leq k\leq n}{\max }\left\Vert x_{n}(t_{k}^{n})\right\Vert )\int\limits_{T_{0}}^{t^{n}_{i+1}}\int\limits_{T_{0}}^{\tau}\beta_{2}(\tau,s)\,ds\,d\tau.\label{es4}
\end{align}
The inequality \eqref{es4} being true for all $i\in\left\{ 0,\cdots,n-1\right\} $, we have
\begin{align*}
\underset{0\leq k\leq n}{\max }\left\Vert x_{n}(t_{k}^{n})\right\Vert&\leq \left\Vert x_{0} \right\Vert +\int\limits_{T_{0}}^{T} \vert \dot{\upsilon}(\tau) \vert\,d\tau + 2(1+\underset{0\leq k\leq n}{\max }\left\Vert x_{n}(t_{k}^{n})\right\Vert )\int\limits_{T_{0}}^{T}\beta_{1}(\tau)\,d\tau\\
&+2(1+\underset{0\leq k\leq n}{\max }\left\Vert x_{n}(t_{k}^{n})\right\Vert )\int\limits_{T_{0}}^{T}\int\limits_{T_{0}}^{\tau}\beta_{2}(\tau,s)\,ds\,d\tau,
\end{align*}
which gives by \eqref{less_beta}
\begin{equation*}
\underset{0\leq k\leq n}{\max }\left\Vert x_{n}(t_{k}^{n})\right\Vert\leq \left\Vert x_{0} \right\Vert +\int\limits_{T_{0}}^{T} \vert \dot{\upsilon}(\tau) \vert\,d\tau+\dfrac{1}{2}(1+\underset{0\leq k\leq n}{\max }\left\Vert x_{n}(t_{k}^{n})\right\Vert ).
\end{equation*}
This can be rewritten as
\begin{equation}\label{es5}
\underset{0\leq k\leq n}{\max }\left\Vert x_{n}(t_{k}^{n})\right\Vert\leq M,
\end{equation}
where $ M:=2\bigg(\left\Vert x_{0} \right\Vert+\int\limits_{T_{0}}^{T} \vert \dot{\upsilon}(\tau) \vert\,d\tau+\dfrac{1}{2}\bigg) $.\\
On one hand, from the growth condition of $f_{1}$, $f_{2}$  and  \eqref{es5} we have, for almost all $t$ and for all $n$,
\begin{equation}\label{e1}
\left\Vert f_{1}(t,x_{n}(\theta _{n}(t)))\right\Vert \leq \beta_{1}(t)(1+\left\Vert x_{n}(\theta _{n}(t))\right\Vert )\leq (1+M)\beta_{1}(t).
\end{equation}
\begin{equation}\label{e2}
\left\Vert f_{2}(t,s,x_{n}(\theta _{n}(s)))\right\Vert\leq \beta_{2}(t,s)(1+\left\Vert x_{n}(\theta _{n}(s))\right\Vert )\leq (1+M)\beta_{2}(t,s).
\end{equation}
Hence, \eqref{es0} implies for almost all $t$ and for all $n$
\begin{align}\label{hyp}
\left\Vert \dot{x}_{n}(t)+f_{1}(t,x_{n}(\theta _{n}(t)))+\displaystyle\int\limits_{T_{0}}^{t}f_{2}(t,s,x_{n}(\theta _{n}(s)))\,ds\right\Vert &\leq (1+M)\bigg(\beta_{1}(t)+\int\limits_{T_{0}}^{t}\beta_{2}(t,s)\,ds\bigg)+\vert \dot{\upsilon}(t) \vert,
\end{align}
and thus
\begin{equation}\label{der}
\left\Vert \dot{x}_{n}(t)\right\Vert\leq 2(1+M)\bigg(\beta_{1}(t)+\int\limits_{T_{0}}^{t}\beta_{2}(t,s)\,ds\bigg)+\vert \dot{\upsilon}(t) \vert.
\end{equation}
\textbf{Step 4. We show that $x_{n}(\cdot)$ converges.} \\
It suffices to show that $x_n(\cdot)$ is a Cauchy sequence in the Banach space
$\left(C(I,H),\left\Vert \cdot\right\Vert_{\infty }\right)$. \ \newline
Let $m,n\in\mathbb{N}$. For almost all $t\in \left[ T_{0},T\right] $, we have
\begin{equation}
\left\{
\begin{array}{c}
-\dot{x}_{n}(t)-f_{1}(t,x_{n}(\theta _{n}(t)))-\displaystyle\int\limits_{T_{0}}^{t}f_{2}(t,s,x_{n}(\theta _{n}(s)))\,ds\in N_{C(t)}(x_{n}(t)), \\
-\dot{x}_{m}(t)-f_{1}(t,x_{m}(\theta _{m}(t)))-\displaystyle\int\limits_{T_{0}}^{t}f_{2}(t,s,x_{m}(\theta _{m}(s)))\,ds\in N_{C(t)}(x_{m}(t)).
\end{array}%
\right.
\end{equation}
Let us set
\begin{gather*}
\alpha(t):=(1+M)\bigg(\beta_{1}(t)+\int\limits_{T_{0}}^{t}\beta_{2}(t,s)\,ds\bigg)+\vert \dot{\upsilon}(t) \vert,\\
 \gamma(t):=2(1+M)\bigg(\beta_{1}(t)+\int\limits_{T_{0}}^{t}\beta_{2}(t,s)\,ds\bigg)+\vert \dot{\upsilon}(t) \vert.
\end{gather*}
The absolute continuity of $x_n(\cdot)$ gives by \eqref{der}
\begin{equation}\label{xn_eta}
\left\Vert x_{n}(t)\right\Vert \leq \eta \text{ for all }t \in \left[ T_{0},T\right],
\end{equation}%
with
\[
\eta :=\left\Vert x_{0}\right\Vert +\int\limits_{T_{0}}^{T}%
\gamma (s)\,ds.
\]%
Using \eqref{hyp} and the hypomonotonicity of the normal cone $N(C(t);\cdot)$, we get that
\begin{align*}
&\langle \dot{x}_{n}(t)+f_{1}(t,x_{n}(\theta _{n}(t)))+\displaystyle\int\limits_{T_{0}}^{t}f_{2}(t,s,x_{n}(\theta _{n}(s)))\,ds-\dot{x}_{m}(t)-f_{1}(t,x_{m}(\theta _{m}(t)))\\
&-\displaystyle\int\limits_{T_{0}}^{t}f_{2}(t,s,x_{m}(\theta _{m}(s)))\,ds, x_{n}(t)-x_{m}(t)\rangle\leq \dfrac{\alpha(t)}{r}\left\Vert x_{n}(t)-x_{m}(t) \right\Vert^{2}
\end{align*}
Therefore
\begin{align*}
&\left\langle \dot{x}_{n}(t)-\dot{x}_{m}(t),x_{n}(t)-x_{m}(t)\right\rangle \leq \dfrac{\alpha(t)}{r}\left\Vert x_{n}(t)-x_{m}(t) \right\Vert^{2}\\
&+\left\langle f_{1}(t,x_{n}(\theta _{n}(t)))-f_{1}(t,x_{m}(\theta _{m}(t))), x_{m}(t)-x_{n}(t) \right\rangle\\
&+\left\langle \displaystyle\int\limits_{T_{0}}^{t}f_{2}(t,s,x_{n}(\theta _{n}(s)))\,ds-\displaystyle\int\limits_{T_{0}}^{t}f_{2}(t,s,x_{m}(\theta _{m}(s)))\,ds , x_{m}(t)-x_{n}(t) \right\rangle.
\end{align*}
Applying the Lipschitz continuity of $f_{1}(t,\cdot)$ and $f_{2}(t,s,\cdot)$ with Lipschitz radius $L_{1}^{\eta}(\cdot),L_{2}^{\eta}(\cdot)\in L^{1}(I,\mathbb{R}_{+})$ on the bounded subset ${B}[0,\eta]$, it follows that
\begin{align*}
&\frac{1}{2}\frac{d}{dt}\left\Vert x_{n}(t)-x_{m}(t)\right\Vert^{2} \leq\dfrac{\alpha(t)}{r}\left\Vert x_{n}(t)-x_{m}(t) \right\Vert^{2}\\
&+ L_{1}^{\eta}(t)\left\Vert x_{n}(t)-x_{m}(t)\right\Vert
\Big( \left\Vert x_{n}(\theta _{n}(t))-x_{n}(t)\right\Vert
 +\left\Vert x_{n}(t)-x_{m}(t)\right\Vert +\left\Vert x_{m}(t)-x_{m}(\theta_{m}(t))\right\Vert \Big)\\
&+L_{2}^{\eta}(t)\left\Vert x_{n}(t)-x_{m}(t)\right\Vert\Big( \int\limits_{T_{0}}^{t}\left\Vert x_{n}(\theta _{n}(s))-x_{n}(s)\right\Vert ds
 +\int\limits_{T_{0}}^{t}\left\Vert x_{n}(s)-x_{m}(s)\right\Vert ds +\int\limits_{T_{0}}^{t}\left\Vert x_{m}(t)-x_{m}(\theta_{m}(t))\right\Vert ds \Big)
\end{align*}
By \eqref{der}, we have for each $n\in\mathbb{N}$ and for all $t$,
\[
\left\Vert x_{n}(t)-x_{n}(\theta _{n}(t))\right\Vert =\left\Vert \int\limits_{\theta _{n}(t)}^{t}\dot{x}_{n}(\tau)d\tau\right\Vert \leq \int\limits_{\theta _{n}(t)}^{t}\left\Vert \dot{x}_{n}(\tau)\right\Vert
d\tau\leq\int\limits_{\theta _{n}(t)}^{t}\gamma (\tau)d\tau.
\]%
Therefore
\begin{align*}
&\frac{1}{2}\frac{d}{dt}\left\Vert x_{n}(t)-x_{m}(t)\right\Vert^{2}\leq\dfrac{\alpha(t)}{r}\left\Vert x_{n}(t)-x_{m}(t) \right\Vert^{2}+L_{1}^{r}(t)\left\Vert x_{n}(t)-x_{m}(t)\right\Vert^{2} \\
&+ L_{1}^{\eta}(t)\left\Vert x_{n}(t)-x_{m}(t)\right\Vert
\Big( \int\limits_{\theta _{n}(t)}^{t}\gamma (\tau)d\tau+\int\limits_{\theta _{m}(t)}^{t}\gamma (\tau)d\tau \Big)\\
&+L_{2}^{\eta}(t)\left\Vert x_{n}(t)-x_{m}(t)\right\Vert\Big( \int\limits_{T_{0}}^{t}\int\limits_{\theta _{n}(s)}^{s}\gamma (\tau)d\tau\,ds+\int\limits_{T_{0}}^{t}\int\limits_{\theta _{m}(s)}^{s}\gamma (\tau)d\tau\,ds \Big)\\
&+L_{2}^{\eta}(t)\left\Vert x_{n}(t)-x_{m}(t)\right\Vert\int\limits_{T_{0}}^{t}\left\Vert x_{n}(s)-x_{m}(s)\right\Vert ds.
\end{align*}
Moreover, noting by $\eqref{xn_eta}$ that
\[
\left\Vert x_{n}(t)-x_{m}(t)\right\Vert \leq \left\Vert x_{n}(t)\right\Vert
+\left\Vert x_{m}(t)\right\Vert \leq 2\eta,
\]
we deduce that
\begin{align}
&\frac{1}{2}\frac{d}{dt}\left\Vert x_{n}(t)-x_{m}(t)\right\Vert^{2}\leq\dfrac{\alpha(t)}{r}\left\Vert x_{n}(t)-x_{m}(t) \right\Vert^{2}+L_{1}^{\eta}(t)\left\Vert x_{n}(t)-x_{m}(t)\right\Vert^{2}\notag \\
&+ 2\eta \L_{1}^{\eta}(t)\Big( \int\limits_{\theta _{n}(t)}^{t}\gamma (\tau)d\tau+\int\limits_{\theta _{m}(t)}^{t}\gamma (\tau)d\tau \Big)+2\eta L_{2}^{\eta}(t)\Big( \int\limits_{T_{0}}^{t}\bigg[\int\limits_{\theta _{n}(s)}^{s}\gamma(\tau)d\tau +\int\limits_{\theta _{m}(s)}^{s}\gamma(\tau)d\tau\bigg]\,ds \Big)\notag\\
&+L_{2}^{\eta}(t)\left\Vert x_{n}(t)-x_{m}(t)\right\Vert\int\limits_{T_{0}}^{t}\left\Vert x_{n}(s)-x_{m}(s)\right\Vert ds.\label{ww}
\end{align}
Let us put
\begin{equation*}
G_{n,m}(t):=2\eta L_{1}^{r}(t)\Big( \int\limits_{\theta _{n}(t)}^{t}\gamma (\tau)d\tau+\int\limits_{\theta _{m}(t)}^{t}\gamma(\tau)d\tau \Big),
\end{equation*}
\begin{equation*}
\widetilde{G}_{n,m}(s):=\int\limits_{\theta _{n}(s)}^{s}\gamma (\tau)d\tau +\int\limits_{\theta _{m}(s)}^{s}\gamma(\tau)d\tau
\end{equation*}
Since $\gamma (\cdot)\in L^{1}(I,\mathbb{R}_{+})$ and for each $t\in I$, we have $\theta _{n}(t),\theta
_{m}(t)\longrightarrow t$, then
\begin{equation}\label{congmn}
\underset{n,m\rightarrow +\infty}{\lim } G_{n,m}(t)=0 \quad\text{and}\quad
%\end{equation}
%\begin{equation}\label{congmnn}
\underset{n,m\rightarrow +\infty}{\lim } \widetilde{G}_{n,m}(t)=0.
\end{equation}
On the other hand, for each $n\in\mathbb{N}$ writing
\[
\int_{\theta _{n}(t)}^{t}\gamma (s)ds\leq \int_{T_{0}}^{T}\gamma (s)ds,
\]%
we see that
\begin{equation*}
\left\vert G_{n,m}(t)\right\vert \leq 4\eta L_{1}^{\eta}(t) \left(
\int_{T_{0}}^{T}\gamma (s)ds\right) \quad\text{and}\quad
% \end{equation*}
% \begin{equation*}
\left\vert \tilde{G}_{n,m}(s)\right\vert \leq 4 \left(
\int_{T_{0}}^{T}\gamma (s)ds\right) .
\end{equation*}
Therefore, for all $t\in \left[ T_{0},T\right]$ by \eqref{congmn}
% \eqref{congmnn}
and the dominated convergence theorem, we obtain
\begin{equation}
\underset{n,m\rightarrow +\infty}{\lim }\int_{T_{0}}^{T}G_{n,m}(t)dt=0.
\label{in2}
\end{equation}
\begin{equation}
\underset{n,m\rightarrow+\infty}{\lim }\int_{T_{0}}^{T}\tilde{G}_{n,m}(s)ds=0.
\label{in22}
\end{equation}
Note also by \eqref{ww} that
\begin{align*}
\frac{1}{2}\frac{d}{dt}\left\Vert x_{n}(t)-x_{m}(t)\right\Vert^{2}&\leq\bigg(\dfrac{\alpha(t)}{r}+L_{1}^{\eta}(t)\bigg)\left\Vert x_{n}(t)-x_{m}(t)\right\Vert^{2} +G_{n,m}(t)+2\eta L_{2}^{\eta}(t)\Big( \int\limits_{T_{0}}^{T}\widetilde{G}_{n,m}(s)\,ds \Big)\\
&+L_{2}^{\eta}(t)\left\Vert x_{n}(t)-x_{m}(t)\right\Vert\int\limits_{T_{0}}^{t}\left\Vert x_{n}(s)-x_{m}(s)\right\Vert ds.
\end{align*}
Applying Lemma \ref{8} with
\begin{gather*}
\rho(t)=\lVert x_{n}(t)-x_{m}(t) \rVert^{2},\,\, K_{1}(t)=2\bigg(\dfrac{\alpha(t)}{r}+L_{1}^{\eta}(t)\bigg),\,\, K_{2}(t)=2L_{2}^{\eta}(t)\\
 \varepsilon(t):=\varepsilon_{n,m}(t)=2G_{n,m}(t)+4\eta L_{2}^{\eta}(t)\Big( \int\limits_{T_{0}}^{T}\widetilde{G}_{n,m}(s)\,ds \Big),\,\, \epsilon>0,
\end{gather*}
 we then see that
\begin{align*}
\lVert x_{n}(t)-x_{m}(t) \rVert\leq &\sqrt{\lVert x_{n}(T_{0})-x_{m}(T_{0}) \rVert^{2}+\epsilon}\,\exp\bigg(\int\limits_{0}^{t}(K(s)+1)\,ds\bigg)\\
&+ \dfrac{\sqrt{\epsilon}}{2}\int\limits_{T_{0}}^{t}\exp\bigg(\int\limits_{s}^{t}(K(\tau)+1)\,d\tau\bigg)\,ds \\
&+2\bigg(\sqrt{\int\limits_{T_{0}}^{t}\varepsilon_{n,m}(s)\,ds+\epsilon} -\exp\bigg(\int\limits_{T_{0}}^{t}(K(\tau)+1)\,d\tau\bigg)\sqrt{\epsilon}\bigg)\\
& + 2\int\limits_{T_{0}}^{t}(K(s)+1)\exp\bigg(\int\limits_{s}^{t}(K(\tau)+1)\,d\tau\bigg)\sqrt{\int\limits_{T_{0}}^{s}\varepsilon_{n,m}(\tau)\,d\tau+\epsilon} \,\,ds.
\end{align*}	
where $ K(t):=\max\bigg\{\dfrac{\alpha(t)}{\eta}+L_{1}^{r}(t),L_{2}^{\eta}(t)\bigg\} $, for almost all $ t\in[T_{0},T] $.\\
This, along with the fact that $ \lVert x_{n}(T_{0})-x_{m}(T_{0}) \rVert=0 $ and taking $ \epsilon\rightarrow 0 $, we get
\begin{equation*}
\underset{n,m\rightarrow+\infty}{\lim }\Vert x_{n}(\cdot)-x_{m}(\cdot) \Vert_{\infty}=0.
\end{equation*}
Therefore, the sequence $(x_{n}(\cdot))$ is a Cauchy sequence in
$(\mathcal{C}(\left[ T_{0},T\right] ,H),\left\Vert\cdot\right\Vert _{\infty })$ and
therefore converges uniformly to a function $x(\cdot)\in \mathcal{C}(\left[ T_{0},T\right] ,H).$ \\
\textbf{Step 5. We show that }$x(\cdot)$ \textbf{is absolutely continuous. }\newline
We have for almost all $t\in I$ and for any $n$,
\begin{equation*}
\left\Vert \dot{x}_{n}(t)\right\Vert \leq\gamma (t).
\end{equation*}
So we can extract a subsequence
of $(\dot{x}_{n}(\cdot))$ (that, without loss of generality, we do not relabel)
which
converges weakly in $L^{1}\left( I,H\right) $ to a function $g(\cdot)\in L^{1}\left( I,H\right).$
This means that
\begin{equation*}
\int_{T_{0}}^{T}\left\langle \dot{x}_{n}(s),h(s)\right\rangle ds\longrightarrow \int_{T_{0}}^{T}\left\langle g(s),h(s)\right\rangle ds,\forall\; h\in L^{\infty}(I,H).
\end{equation*}%
Now observe that for all $z\in H$
\[
\int_{T_{0}}^{T}\left\langle \dot{x}%
_{n}(s),z\cdot\mathrm{1}_{[T_{0},t] }(s)\right\rangle ds=\int_{T_{0}}^{t}\left\langle \dot{x}_{n}(s),z\right\rangle ds=\langle z,\int_{T_{0}}^{t}\dot{x}_{n}(s)ds\rangle.
\]
and
\[
\int_{T_{0}}^{T}\left\langle g(s),z\cdot\mathrm{1}_{[
T_{0},t] }(s)\right\rangle ds=\int_{T_{0}}^{t}
\left\langle g(s),z\right\rangle ds=\langle z,\int_{T_{0}}^{t}
g(s)ds\rangle.
\]%
So from the weak convergence we deduce that
\[
\int_{T_{0}}^{t}\dot{x}_{n}(s)ds \longrightarrow \int_{T_{0}}^{t}g(s)ds \text{ weakly in } H.
\]
This implies that
\[ x_{n}(T_{0})+\int_{T_{0}}^{t}\dot{x}_{n}(s)ds\longrightarrow x(T_{0})+\int_{T_{0}}^{t}g(s)ds%
\text{ weakly in }H.
\]%
But $x_{n}(\cdot)$ is absolutely continuous, so
\[
x_{n}(t)=x_{n}(T_{0})+\int_{T_{0}}^{t}\dot{x}_{n}(s)ds\longrightarrow x(T_{0})+\int_{T_{0}}^{t}g(s)ds
\text{ weakly in }H.
\]%
On the other hand, we have for all $t\in [ T_{0},T] $
\[
x_{n}(t)\longrightarrow x(t)\text{ strongly in }H,
\]%
% This implies that
%  \[
%  x_{n}(t)\longrightarrow x(t)\text{ weakly in }H.
%  \]%
 % From the uniqueness of the limit, we have
hence we get
\[
x(t)=x(T_{0})+\int_{T_{0}}^{t}g(s)ds.
\]%
Therefore, $x(\cdot)$ is absolutely continuous
%  and moreover, it is derivable almost everywhere in $\left[ T_{0},T\right] $ with its
%  derivative
and $\dot{x}(t)=g(t)$ a.e. $t\in \left[ T_{0},T\right] $, so in particular
\begin{equation}\label{xn_etaa}
\left\Vert x(t)\right\Vert \leq \tilde{\eta}\text{ for all }t\in \left[ T_{0},T\right],
\end{equation}%
with
\[
\tilde{\eta} :=\left\Vert x_{0}\right\Vert +\int\limits_{T_{0}}^{T}%
g(s)\,ds.
\]\ \ \newline
\textbf{Step 6. We show that $x(\cdot)$ is a solution of $(P_{f_{1},f_{2}})$ .} \newline
For each $t\in I$, since $\theta _{n}(t)\longrightarrow t$ for all $t\in I$ and $x_{n}(\cdot)$
converges uniformly to $x(\cdot)$, we have $x_{n}(\theta_{n}(t))\longrightarrow x(t)$.\\
Let us set for each $t\in I$
% $ (t,s)\in P_{\Delta} $
\begin{equation*}
y_{n}(t):=\int\limits_{T_{0}}^{t}f_{2}(t,s,x_{n}(\theta _{n}(s)))ds,\,\,\text{and}\,\,y(t):=\int\limits_{T_{0}}^{t}f_{2}(t,s,x(s))ds.
\end{equation*}

We have shown in the above step that $\dot{x}_{n}(\cdot)$ converges weakly to $\dot{x}(\cdot)$ in $L^{1}(I,H)$. \newline
Moreover, by \eqref{xn_eta} and \eqref{xn_etaa} we can choose some
real $c>0$ such that, for each $n$, $\lVert x_{n}(\theta_{n}(t))\rVert,\lVert x(t)\rVert\leq c$ for all $t\in [T_{0},T]$. Therefore, by assumption, there exists\\ $ L^{c}_{1}(\cdot),L^{c}_{2}(\cdot)\in L^{1}([T_{0},T],\mathbb{R}_{+}) $ such that $ f_{1}(t,\cdot) $ and $ f_{2}(t,s,\cdot) $ are $ L^{c}_{1}(t) $-Lipschitz and $ L^{c}_{2}(t) $-Lipschitz respectively on $ B[0,c] $. It follows that
\begin{align}\label{cf}
\int\limits_{T_{0}}^{T}\left\Vert f_{1}(t,x_{n}(\theta _{n}(t)))-f_{1}(t,x(t)) \right\Vert dt \leq \int\limits_{T_{0}}^{T} L^{c}_{1}(t)\left\Vert x_{n}(\theta _{n}(t))-x(t) \right\Vert dt
\end{align}
\begin{align}\label{cff}
\int\limits_{T_{0}}^{T}\left\Vert y_{n}(t)-y(t) \right\Vert dt \leq \int\limits_{T_{0}}^{T} L^{c}_{2}(t)\int\limits_{T_{0}}^{t}\left\Vert x_{n}(\theta _{n}(s))-x(s) \right\Vert ds\, dt.
\end{align}
Note that for every  $ (t,s)\in Q_{\Delta} $
\begin{equation*}
L^{c}_{1}(t)\left\Vert x_{n}(\theta _{n}(t))-x(t) \right\Vert\leq 2cL^{c}_{1}(t),
\end{equation*}
\begin{equation*}
L^{c}_{2}(t)\int\limits_{T_{0}}^{t}\left\Vert x_{n}(\theta _{n}(s))-x(s) \right\Vert ds\leq 2c(T-T_{0})L^{c}_{2}(t).
\end{equation*}
Then by \eqref{cf}, \eqref{cff} and by the Lebesgue dominated convergence theorem
\[
f_{1}(\cdot,x_{n}(\theta _{n}(\cdot)))\longrightarrow f_{1}(\cdot,x(\cdot))\text{ strongly in }
L^{1}(I,H).
\]
\[
y_{n}(\cdot)\longrightarrow y(\cdot) \text{ strongly in }
L^{1}(I,H).
\]
This implies that
\[
\zeta _{n}(\cdot):=\dot{x}_{n}(\cdot)+f_{1}(\cdot,x_{n}(\theta _{n}(\cdot)))+y_{n}(\cdot)\longrightarrow \zeta(\cdot):= \dot{x
}(\cdot)+f_{1}(\cdot,x(\cdot))+y(\cdot)
\]weakly in $ L^{1}(I,H) $.\\
By Mazur's lemma we can find a convex combination $ \sum\limits_{k=n}^{r(n)}S_{k,n}\zeta _{k}(\cdot) $, with $ \sum\limits_{k=n}^{r(n)}S_{k,n}=1 $ and $ S_{k,n}\in[0,1] $ for all $ k,n $,  which converges strongly in $ L^{1}(I,H) $ to $ \zeta(\cdot)  $. Extracting a subsequence, we may suppose that $ \sum\limits_{k=n}^{r(n)}S_{k,n}\zeta _{k}(\cdot) $ converges almost everywhere on $ I $ to some mapping $ \zeta(\cdot) $.\\
Further, we know that there is a negligible set $N \subset I$ such that for each $t\in I \setminus N$ one has for all $n\in \mathbb{N}$
% For all $t \in I-N $ for some negligible set $N \subset I$, fixed $t \in  I - N$ , one has
\begin{equation*}
-\zeta _{n}(t):=-\dot{x}_{n}(t)-f_{1}(t,x_{n}(\theta _{n}(t)))-\displaystyle\int\limits_{T_{0}}^{t}f_{2}(t,s,x_{n}(\theta _{n}(s)))\,ds\in N_{C(t)}(x_{n}(t)).
\end{equation*}
Fix any $t\in I\setminus N$ and any $n\in \mathbb{N}$. From  Definition \ref{def} of the normal cone, one has for every $z\in C(t)$
\begin{equation*}
\langle -\zeta _{n}(t) , z-x_{n}(t) \rangle\leq \dfrac{\gamma(t)}{2r}\Vert z-x_{n}(t) \Vert^{2}\,\,\,\,\text{for all}\,\,\,z\in C(t),
\end{equation*}
hence
\begin{equation}\label{18}
\langle -\zeta _{n}(t) , z-x_{n}(t) \rangle\leq \dfrac{\gamma(t)}{2r}(\Vert z-x(t) \Vert+\Vert x(t)-x_{n}(t) \Vert)^{2}:=\lambda_{n}(t),
\end{equation}
with $ \lim\limits_{n\longrightarrow \infty}\lambda_{n}(t)=\dfrac{\gamma(t)}{2r}\lVert z-x(t) \rVert^{2} $. Therefore,
 \begin{align*}
 \langle -\zeta(t),z-x(t) \rangle
 &=  \langle -\zeta(t)+\sum\limits_{k=n}^{r(n)}S_{k,n}\zeta _{k}(t),z-x(t) \rangle
 +\sum\limits_{k=n}^{r(n)}S_{k,n}\langle -\zeta _{k}(t),z- x_{k}(t) \rangle\\
 &+\sum\limits_{k=n}^{r(n)}S_{k,n}\langle -\zeta _{k}(t) ,-x(t)+ x_{k}(t) \rangle.
 \end{align*}
 The first expression of the second member of the latter equality tends to zero by what precedes, and keeping in mind that $|\zeta_k(t)|\leq \gamma(t)$, we also see that the third expression tends to zero.
% and the third summands in the above expression tend to zero a.e. $ t\in I $.
Concerning the second expression, thanks to \eqref{18}, it satisfies the estimate
 \begin{align*}
 \sum\limits_{k=n}^{r(n)}S_{k,n}\langle -\zeta _{k}(t),z- x_{k}(t) \rangle\leq \sum\limits_{k=n}^{r(n)}S_{k,n}\lambda_{k}(t).
 \end{align*}
 Thus, passing to the limit we obtain
%  and by the a.e. convergence of $ \lambda_{k}(t) $, one obtains
 \begin{equation*}
 \langle -\zeta(t),z-x(t) \rangle\leq \dfrac{\gamma(t)}{2r}\lVert z-x(t) \rVert^{2}, \,\,\,\,\,\,\,\,\forall \,  z\in C(t).
 \end{equation*}
 This proves that
 \begin{equation*}
  -\dot{x}(t)-f_{1}(t,x(t))-\int\limits_{T_{0}}^{t}f_{2}(t,s,x(s))ds \in N_{C(t)}(x(t)),\,\,\,  a.e.\,\,\,  t\in[T_{0},T],
 \end{equation*}
  and thus
  \begin{equation*}
  -\dot{x}(t) \in N_{C(t)}(x(t))+f_{1}(t,x(t))+\int\limits_{T_{0}}^{t}f_{2}(t,s,x(s))ds ,\,\,\,  a.e.\,\,\,  t\in[T_{0},T].
  \end{equation*}
	
   Now consider the situation when
  \begin{equation*}
\int\limits_{T_0}^T\bigg[\beta_{1}(\tau)+\int\limits_{T_0}^{\tau}\beta_{2}(\tau,s)\,ds\bigg] d\tau\geq\dfrac{1}{4}.
\end{equation*}
We  fix a subdivision of $ [T_{0},T] $ given by $ T_{0},T_{1},...,T_{k}=T $ such that, for any\\ $ 0\leq i\leq k-1 $,
\begin{equation*}
\int\limits_{T_i}^{T_{i+1}}\bigg[\beta_{1}(\tau)+\int\limits_{T_0}^{\tau}\beta_{2}(\tau,s)\,ds\bigg] d\tau<\dfrac{1}{4}.
\end{equation*}
Then, by what precedes, there exists an absolutely continuous map $ x_{0} : [T_0,T_{1}]\longrightarrow H $ such that $ x_{0}(T_0)=x_{0} $, $ x_{0}(t)\in C(t) $ for all $ t\in [T_0,T_{1}] $, and
    \begin{equation*}
 -\dot{x}_{0}(t) \in N_{C(t)}(x_{0}(t))+f_{1}(t,x_{0}(t))+\int\limits_{T_{0}}^{t}f_{2}(t,s,x_{0}(s))\,ds,\,\,\,  a.e.\,\,\,  t\in[T_0,T_{1}] .
    \end{equation*}
    Similarly, there is  an absolutely continuous map $ x_{1} : [T_{1},T_{2}]\longrightarrow H $ such that\\ $ x_{1}(T_{1})=x_{0}(T_{1}) $, $ x_{1}(t)\in C(t) $ for all $ t\in [T_{1},T_{2}] $, and
    \begin{equation*}
    -\dot{x}_{1}(t) \in N_{C(t)}(x_{1}(t))+f_{1}(t,x_{1}(t))+\int\limits_{T_{0}}^{t}f_{2}(t,s,x_{1}(s))\,ds,\,\,\,  a.e.\,\,\,  t\in[T_{1},T_{2}].
    \end{equation*}
    By induction, we obtain for each $ 0\leq i \leq k-1 $ a finite sequence of absolutely continuous maps $  x_{i} : [T_{i},T_{i+1}]\longrightarrow H  $ such that for each $ 0\leq i \leq k-1 $, $ x_{i}(T_{i})=x_{i-1}(T_{i}) $ and $ x_{i}(t)\in C(t) $ for all $ t\in [T_{i},T_{i+1}] $, and
    \begin{equation*}
    -\dot{x}_{i}(t) \in N_{C(t)}(x_{i}(t))+f_{1}(t,x_{i}(t))+\int\limits_{T_{0}}^{t}f_{2}(t,s,x_{i}(s))\,ds,\,\,\,  a.e.\,\,\,  t\in[T_{i},T_{i+1}] .
    \end{equation*}
    We set $ x_{-1}(0)=x_{0} $ and define the mapping $ x : [T_{0},T]\longrightarrow H $ given by
    \begin{equation*}
    x(t)=x_{i}(t),\,\,\,\\text{if}\,\,\,t\in[T_{i},T_{i+1}],\,\,\,0\leq i\leq k-1 .
    \end{equation*}
    Obviously, $ x(\cdot) $ is an absolutely continuous mapping satisfying $ x(T_{0})=x_{0} $, $ x(t)\in C (t) $ for all $ t\in[T_{0},T] $ and
    \begin{equation}\label{34}
    -\dot{x}(t) \in N_{C(t)}(x(t))+f_{1}(t,x(t))+\int\limits_{T_{0}}^{t}f_{2}(t,s,x(s))\,ds,\,\,\,  a.e.\,\,\,  t\in[T_{0},T] .
    \end{equation}
  \textbf{Step 7. We prove the estimations. }\newline
Let $x(\cdot)$ be a solution of $(P_{f_{1},f_{2}})$.\\
Take $N\subset [T_0,T]$ such that $\lambda(N)=0$ and the inclusion \eqref{34} holds for every $t\in[T_0,T]\setminus N$. Fix any $t\in [T_0,T]\setminus N$.

 By definition of proximal normal cone, there is some real $a_0>0$ such that for
any $a\in (0,a_0]
$
\begin{equation*}
x(t)\in \mathrm{Proj}_{C(t)}(x(t) -a \dot{x}(t)-a f_{1}(t,x(t))-a\int\limits_{T_0}^{t}f_{2}(t,s,x(s))\,ds).
\end{equation*}
 We derive from the latter inclusion that
\begin{align*}
  & a\lVert  \dot{x}(t)+ f_{1}(t,x(t))+\int\limits_{T_0}^{t}f_{2}(t,s,x(s))\,ds \rVert =\displaystyle d_{C(t)}\Big(x(t) -a \dot{x}(t)-a\varepsilon f_{1}(t,x(t))
-a\int\limits_{T_0}^{t}f_{2}(t,s,x(s))\,ds\Big)\\
&\leq |\upsilon(t)-\upsilon(\tau)|+\Big\| x(t)-x(\tau)-a \dot{x}(t)
  -a f_{1}(t,x(t))-a\int\limits_{T_0}^{t}f_{2}(t,s,x(s))\,ds\Big\|,
\end{align*}
since $x(\tau)\in C(\tau)$ for all $\tau\in[T_0,T]$. For any $\tau\in T_0,t[$ with
$t-a_0< \tau <t$, taking $a=t-\tau$  one obtains
\begin{align*}
 & \Big\|  \dot{x}(t)+ f_{1}(t,x(t))+\int\limits_{T_0}^{t}f_{2}(t,s,x(s))\,ds \Big\|\\
 &\leq \dfrac{\rvert\upsilon(t)-\upsilon(\tau)\lvert}{ t-\tau }+\Big\| \dfrac{x(t)-x(\tau)}
{t-\tau}- \dot{x}(t) -f_{1}(t,x(t))-\int\limits_{T_0}^{t}f_{2}(t,s,x(s))\,ds\Big\|.
\end{align*}
Making $\tau \uparrow t$ yields
\begin{align}
\Big\|  \dot{x}(t)+ f_{1}(t,x(t))+\int\limits_{T_0}^{t}f_{2}(t,s,x(s))\,ds \Big\| &\leq
\Big\|\dot{\upsilon}(t)\rvert+\lVert  -f_{1}(t,x(t))-\int\limits_{T_0}^{t}f_{2}(t,s,x(s))\,ds \Big \|\notag\\
&\leq \lvert\dot{\upsilon}(t)\rvert + \lVert  f_{1}(t,x(t)) \rVert + \int\limits_{T_0}^{t}\lVert f_{2}(t,s,x(s)) \rVert \,ds.\label{35}
\end{align}
This justifies  \eqref{30}.

   Now assume
\begin{equation*}
\int\limits_{T_0}^{t}\bigg[\beta_{1}(\tau)+\int\limits_{T_0}^{\tau}\beta_{2}(\tau,s)\,ds\bigg] d\tau<\dfrac{1}{4}.
\end{equation*}
We have from \eqref{e1}, \eqref{e2} and \eqref{hyp} that the estimates \eqref{31}, \eqref{32} and \eqref{33} are obviously fulfilled.\\
If in addition
\begin{equation*}
\lVert f_{2}(t,s,x) \rVert\leq  g(t,s) +\alpha(t)\lVert x \rVert
\end{equation*}
 we have from \eqref{35} that
\begin{align}
\lVert  \dot{x}(t) \rVert &\leq \rvert\dot{\upsilon}(t)\rvert +2 \lVert  f_{1}(t,x(t)) \rVert + 2\int\limits_{T_0}^{t}\lVert f_{2}(t,s,x(s)) \rVert \,ds\notag\\
&\leq \rvert\dot{\upsilon}(t)\rvert +2\beta_{1}(t)(1+\lVert x(t) \rVert)+2\int\limits_{T_0}^{t}g(t,s)\,ds+2\alpha(t)\int\limits_{T_0}^{t}\lVert x(s) \rVert\,ds\notag\\
&=\rvert\dot{\upsilon}(t)\rvert +2\beta_{1}(t)+2\int\limits_{T_0}^{t}g(t,s)\,ds+2\beta_{1}(t)\Vert x(t) \Vert+2\alpha(t)\int\limits_{T_0}^{t}\lVert x(s) \rVert\,ds. \label{mmmm}
\end{align}
Putting $\rho(t):=\Vert x_{0} \Vert +\displaystyle\int\limits_{T_0}^{t}\Vert \dot{x}(s) \Vert \,ds $ and noting that
 $ \Vert x(t) \Vert\leq \rho(t) $,  the inequality \eqref{mmmm} ensures  that
 \begin{equation*}
 \dot{\rho}(t)\leq \rvert\dot{\upsilon}(t)\rvert +2\beta_{1}(t)+2\int\limits_{T_0}^{t}g(t,s)\,ds+2\beta_{1}(t)\rho(t)+2\alpha(t)\int\limits_{T_0}^{t}\rho(s)\,ds.
 \end{equation*}
Applying Gronwall Lemma \ref{22} with $ \rho(\cdot) $, one obtains
\begin{align*}
\lVert x(t) \rVert\leq \rho(t) &\leq \lVert x_{0} \rVert\exp\bigg(\int\limits_{T_0}^{t}(b(\tau)+1)\,d\tau\bigg)\\
&+\int\limits_{T_0}^{t}\bigg(\rvert\dot{\upsilon}(s)\rvert +2\beta_{1}(s)+2\int\limits_{T_0}^{s}g(s,\tau)\,d\tau\bigg)\exp\bigg(\int\limits_{s}^{t}(b(\tau)+1)\,d\tau\bigg)\,ds,
\end{align*}
where $b(\tau):=2\max\{\beta_{1}(\tau),\alpha(\tau)\}$ for almost all $\tau\in [T_0,T]$. This yields the validity of \eqref{41}, \eqref{42} and  \eqref{43}.\\
\textbf{Step 8. Uniqueness.}\newline
Now, we turn to the uniqueness. If $ x_{1}(\cdot), x_{2}(\cdot) $ are two solutions, the hypo-monotonicity property of the normal cone yields for almost all $ t \in[T_0,T] $
 \begin{align*}
& \langle -\dot{x}_{1}(t)-f_{1}(t,x_{1}(t))-\int\limits_{T_0}^{t} f_{2}(t,s,x_{1}(s))\,ds+\dot{x}_{2}(t)+f_{1}(t,x_{2}(t))+\int\limits_{T_0}^{t} f_{2}(t,s,x_{2}(s))\,ds,\\
& x_{2}(t)-x_{1}(t)\rangle\leq \dfrac{1}{2r}\lVert x_{2}(t)-x_{1}(t)  \rVert^{2}\sum\limits_{i=1}^{2} \bigg(\lVert  \dot{x}_{i}(t) \rVert+\lVert  f_{1}(t,x_{i}(t)) \rVert + \int\limits_{T_0}^{t}\lVert f_{2}(s,x_{i}(s)) \rVert \,ds \bigg),
 \end{align*}
 from which we obtain
  \begin{align*}
 & \langle \dot{x}_{2}(t)-\dot{x}_{1}(t) , x_{2}(t)-x_{1}(t)\rangle\\
 &\leq\dfrac{1}{2r}\lVert x_{2}(t)-x_{1}(t)  \rVert^{2}\sum\limits_{i=1}^{2} \bigg(\lVert  \dot{x}_{i}(t) \rVert+\lVert  f_{1}(t,x_{i}(t)) \rVert + \int\limits_{T_0}^{t}\lVert f_{2}(t,s,x_{i}(s)) \rVert \,ds \bigg)\\&
+ \langle f_{1}(t,x_{1}(t))-f_{1}(t,x_{2}(t)) , x_{2}(t)-x_{1}(t) \rangle\\
  &+ \langle \int\limits_{T_0}^{t} f_{2}(t,s,x_{1}(s))\,ds-\int\limits_{T_0}^{t} f_{2}(t,s,x_{2}(s))\,ds , x_{2}(t)-x_{1}(t) \rangle.
  \end{align*}
  Since the absolutely continuous mappings $x_{1}(\cdot)$ and $x_{2}(\cdot)$ are in particular bounded on $[T_0,T]$, we can choose some
real $\eta>0$ such that, for each $i=1,2$, $\lVert x_{i}(t)\rVert\leq \eta$ for all $t\in [T_0,T]$. The latter inequality assures us that
  \begin{align*}
  &\dfrac{d}{dt}\dfrac{1}{2}\lVert x_{2}(t)-x_{1}(t)  \rVert^{2}\leq L_{2}^{\eta}(t)\lVert x_{2}(t)-x_{1}(t)  \rVert\int\limits_{T_0}^{t}\lVert x_{2}(s)-x_{1}(s)  \rVert\,ds \\
&+\bigg(L_{1}^{\eta}(t)+\dfrac{1}{2r}\sum\limits_{i=1}^{2} \bigg(\lVert  \dot{x}_{i}(t) \rVert+\lVert  f_{1}(t,x_{i}(t)) \rVert + \int\limits_{T_{0}}^{t}\lVert f_{2}(t,s,x_{i}(s)) \rVert \,ds \bigg)\bigg)\lVert x_{2}(t)-x_{1}(t)  \rVert^{2}.
  \end{align*}
  Finally, setting $ \rho(t):=\lVert x_{2}(t)-x_{1}(t)  \rVert^{2} $ we get
  \begin{align*}
  \dot{\rho}(t)&\leq \bigg(2L_{1}^{\eta}(t)+\dfrac{1}{r}\sum\limits_{i=1}^{2} \bigg(\lVert  \dot{x}_{i}(t) \rVert+\lVert  f_{1}(t,x_{i}(t)) \rVert + \int\limits_{T_0}^{t}\lVert f_{2}(t,s,x_{i}(s)) \rVert \,ds \bigg)\bigg)\rho(t)\\
  &+2L_{2}^{\eta}(t)\sqrt{\rho(t)}\int\limits_{T_0}^{t}\sqrt{\rho(s)}\,ds,
  \end{align*}
hence it suffices to invole Lemma \ref{8}  with  $ \varepsilon(\cdot),\epsilon> 0 $ arbitrary. Then the proof  of the theorem is complete .
\end{proof}
\begin{proposition}\label{stability}
 Assume that the assumptions of Theorem \ref{exist} (  in case $ 3 $ ) holds. For each $a\in C(T_{0})$,
denote by $x_{a}(\cdot)$ the unique solution of the integro-differential   sweeping process
\begin{equation*}
\left\{
\begin{array}{l}
-\dot{x}(t) \in N_{C(t)}(x(t))+f_{1}(t,x(t))+\displaystyle\int\limits_{T_{0}}^{t}f_{2}(t,s,x(s))\,ds\quad a.e \; in \; [T_{0},T]\\
x(T_{0})=a\in C(T_{0})
\end{array}\right.
\end{equation*}
Then, the map  $ \psi : a \longrightarrow x_{a}(\cdot) $  from $C(T_{0})$ to the space $C([T_{0},T],H)$ endowed with the
uniform convergence norm is Lipschitz on any bounded subset of $C(T_{0})$.
\end{proposition}
\begin{proof}
%Similarly with to the proof of Proposition 2 in \cite{edm} and using the Gronwall-like differential inequality in Lemma \ref{8} .
Let $M$ be any fixed positive real number. We are going to prove that $\psi$ is Lipschitz
on $C(T_{0})\cap M\mathbb{B}$.\\
According to Theorem \ref{exist} ( case $ 3 $ ), there exists a real number $M_{1}$ depending only on $M$ such
that, for all $z\in C(T_{0})\cap M\mathbb{B} $ and for almost all $(t,s)\in Q_{\Delta}$
\begin{equation*}
\lVert \dot{x}_{z}(t)+f_{1}(t,x_{z}(t))+\int\limits_{T_{0}}^{t}f_{2}(t,s,x_{z}(s))\,ds \rVert\leq \alpha(t):= \rvert\dot{\upsilon}(t)\rvert+ (1+M_{1})\beta_{1}(t)+\int\limits_{T_0}^{t}g(t,s)\,ds+T\alpha(t)M_{1},
\end{equation*}
Thanks to this last inequality, for some $\eta > 0$ depending only on $M$, for all $z\in C(T_{0})\cap M\mathbb{B} $ and for all $t\in [T_{0},T]$, we have
\begin{equation}\label{kk1}
x_{z}(t)\in B[0,\eta] .
\end{equation}
Fix any $a,b\in C(T_{0})\cap M\mathbb{B} $ . By the hypomonotonicity property of the normal cone, we have for almost all $(t,s)\in Q_{\Delta}$
\begin{align*}
& \langle -\dot{x}_{a}(t)-f_{1}(t,x_{a}(t))-\int\limits_{T_{0}}^{t} f_{2}(t,s,x_{a}(s))\,ds+\dot{x}_{b}(t)+f_{1}(t,x_{b}(t))+\int\limits_{T_{0}}^{t} f_{2}(t,s,x_{b}(s))\,ds, x_{2}(t)-x_{1}(t)\\
 &\leq \dfrac{\alpha(t)}{r}\lVert x_{b}(t)-x_{a}(t)  \rVert^{2},
 \end{align*}
 from which we obtain
  \begin{align*}
  \langle \dot{x}_{b}(t)-\dot{x}_{a}(t) , x_{b}(t)-x_{a}(t)\rangle&\leq \dfrac{\alpha(t)}{r}\lVert x_{b}(t)-x_{a}(t)  \rVert^{2}
+ \langle f_{1}(t,x_{a}(t))-f_{1}(t,x_{b}(t)) , x_{b}(t)-x_{a}(t) \rangle\\
  &+ \langle \int\limits_{T_{0}}^{t} f_{2}(t,s,x_{a}(s))\,ds-\int\limits_{T_{0}}^{t} f_{2}(t,s,x_{b}(s))\,ds , x_{b}(t)-x_{a}(t) \rangle.
  \end{align*}
Since, by assumption $ \mathcal{(}\mathcal{H}_{2,2}\mathcal{)} $ and $ \mathcal{(}\mathcal{H}_{3,2}\mathcal{)} $, there are a non-negative functions $L_{1}^{\eta}(\cdot)$, $L_{2}^{\eta}(\cdot)$ $ \in L^{1}([T_{0},T],\mathbb{R}) $  such
that $f_{1}(t,\cdot)$ and $f_{2}(t,s,\cdot)$ are $L_{1}^{\eta}(\cdot)$-Lipschitz, $L_{2}^{\eta}(\cdot)$-Lipschitz (respectively) on $B[0,\eta] $, the above
inequality along with \eqref{kk1}, entails that for almost all $t\in [T_{0},T]$,
\begin{align*}
  \dfrac{d}{dt}\lVert x_{b}(t)-x_{a}(t)  \rVert^{2}&\leq 2\bigg(L^{\eta}_{1}(t)+\dfrac{\alpha(t)}{r} \bigg)\lVert x_{b}(t)-x_{a}(t)  \rVert^{2}+2L_{2}^{\eta}(y)\lVert x_{b}(t)-x_{a}(t)  \rVert\int\limits_{T_{0}}^{t}\lVert x_{b}(s)-x_{a}(s)  \rVert\,ds.
  \end{align*}
Applying Gronwall-like differential inequality in Lemma \ref{8} , it results that
\begin{equation*}
\sup\limits_{t\in [0,T]}\lVert x_{b}(t)-x_{a}(t)  \rVert \leq \lVert b-a \rVert \exp\bigg(\int\limits_{T_{0}}^{t}(K(s)+1)\,ds\bigg),
\end{equation*}
where $ K(t):=\max\bigg\{L_{1}^{\eta}(t)+\dfrac{\alpha(t)}{r},L_{2}^{\eta}(t)\bigg\} $, for almost all $ t\in[T_{0},T] $. The proof is then complete.
\end{proof}
%%%%%%%%%%%%%%%%%%%%%%%%%%%%%%%%%%%%%%%%%%%%%%%%%%%%%%%%%%%%%%%%%%%%%%%%%%%%%%%%%%%%%%%%%%%%%%%%%%%%%%%%%%%%%%%%%%%%%%%%%%%%%%%%%%%%%%%%%%%%%%%%%%%%%%%%%%%%%%%%%%%%%%%%%%%%%%%%%%%%%%%%%%%%%%%%%%%%%%%%%%%%%%%%%%%%%%%%%%%%%%%%%%%%%%%%%%%%%%%%%%%%%%%%%%%%%%%%%%%%%%%%%%%%%%%%%%%%%%%%
\section{Nonlinear integro-differential complementarity systems}\label{s5}\ \\
In this section, as a consequence of Theorem \ref{exist}, we obtain the existence and uniqueness of
solutions for nonlinear integro-differential complementarity systems. Our results generalize those from \cite{an}.\\

Let $T>T_0$ be real numbers, $I=[T_0,T]$, $n,m\in\mathbb{N}$, $f _{1}: I\times\mathbb{R}^{n}\longrightarrow \mathbb{R}^{n}  $, $f _{2}: Q_{\Delta}\times\mathbb{R}^{n}\longrightarrow \mathbb{R}^{n}  $  and $g : I\times\mathbb{R}^{n}\longrightarrow \mathbb{R}^{m} $ be given mappings. Assuming that $g(t,\cdot)$ is differentiable for each $t\in I$, the NIDCS (associated with $f_1$, $f_2$ and $g$)
can be described as
\begin{equation*}
\text{( NIDCS ):} \left\{
\begin{array}{l}
-\dot{x}(t) =f_{1}(t,x(t))+\displaystyle\int\limits_{T_0}^{t}f_{2}(t,s,x(s))\,ds+\nabla\,g(t,\cdot)(x(t))^{T}z(t)\quad \\
0\leq z(t) \perp g(t,x)\leq 0,
\end{array}\right.
\end{equation*}
where $z : I\longrightarrow \mathbb{R}^{m} $ is unknown mapping. The term $\nabla\,g(t,\cdot)(x(t))^{T}z(t)$  can be seen as the generalized
reactions due to the constraints in mechanics.\\
Of course, the behaviour of a solution with respect to $t$ is connected to the variation with respect to $t$ of the set constraint
\begin{equation}\label{e}
C(t)=\{ x\in \mathbb{R}^{n} : g_{1}(t,x)\leq 0,  g_{2}(t,x)\leq 0,..., g_{m}(t,x)\leq 0 \},
\end{equation}
where we set $g(t, \cdot) = (g_1(t, \cdot),g_2(t, \cdot),..., g_m(t, \cdot))$ for each $t \in I.$
\begin{theorem}\cite{an}\label{61}
Let $C(t)$ be defined as in \eqref{e} and assume that, there exists an extended real $ \rho\in ]0,\infty]$ such that
\begin{enumerate}
\item for all $t\in I$, for all $k\in\{1,...,m\}$, $g_{k}(t,\cdot)$ is continuously differentiable on $U_{\rho}(C(t)):=\{y\in\mathbb{R}^{n}:\,\,d_{C(t)}(y)<\rho\}$;\\
\item there exists a real $\gamma>0$ such that, for all $t\in I$, for all $k\in\{1,...,m\}$, $g_{k}(t,\cdot)$, for all $x,y\in U_{\rho}(C(t)) $
\begin{equation*}
\langle \nabla g_{k}(t,\cdot)(x)-\nabla g_{k}(t,\cdot)(y),x-y \rangle\geq -\gamma \lVert x-y \rVert^{2},
\end{equation*}
that is, $\nabla g_{k}(t,\cdot)$ is $\gamma$-hypomonotone on $U_{\rho}(C(t))$;\\
\item there is a real $\delta>0$
 such that for all $(t,x)\in I\times\mathbb{R}^{n}$ with $x\in\, bdry\, (C(t))$, there exists $\bar{\upsilon}\in\mathbb{B}$ satisfying, for all $k\in\{1,...,m\}$
\begin{equation}
\langle \nabla g(t,\cdot)(x),\bar{\upsilon} \rangle\leq -\delta .
\end{equation}
\end{enumerate}
Then for all $t\in I$, the set $C(t)$ is r-prox-regular with $r=\min\{\rho, \dfrac{\delta}{\gamma}\}$ .
\end{theorem}
The nonlinear differential complementarity systems (NDCS) (i.e., (NIDCS) with $f_{2} \equiv 0$ ) was studied in \cite{an}, where the authors transform the (NDCS) involving inequality constraints $C(t)$ to a perturbed sweeping process %of the form (\ref{eq0.2})
. We extend this approach
by transforming (NIDCS) into an integro-differential sweeping process of the form (\ref{eq1.1-N}). Also, in contrast to  \cite{an}, we do not assume that the moving set $C(t)$ described by a finite number of inequalities is absolutely continuous with respect to the Hausdorff distance. Rather, we provide sufficient verifiable conditions ensuring this regularity needed on $C(\cdot)$.
\begin{proposition}\label{60}
 Let $C(t)$ be defined as in \eqref{e}. Assume that there exist an absolutely continuous function $w$, a real $\delta>0$ and a vector $y\in \mathbb{R}^{n}$ with $\lVert y \rVert=1$ such that for each $i=1,...m$
\begin{equation}\label{70}
g_{i}(t,x)\leq g_{i}(s,x) + \lvert w(t)-w(s)  \rvert , \, \,\,\, for \, all \,\, x\in U_{r}(C(s)),
\end{equation}
\begin{equation}\label{50}
\langle \nabla  g_{i}(t,\cdot)(x),y \rangle\leq -\delta , \,\, for, all\,\, t\in I,\, x\in U_{r}(C(t)),
\end{equation}
where $r$ denotes the prox-regularity constant of all sets $C(t)$.
Then $C(\cdot)$ is $\upsilon(\cdot)-$ absolutely continuous on $I$ with
$\upsilon(\cdot):=\delta^{-1}w(\cdot)$.
%  $\dfrac{\lvert w(t)-w(s)  \rvert}{\delta}\leq\lvert \upsilon(t)-\upsilon(s)  \rvert$,  for
% all $t,s\in I$ .
\end{proposition}
\begin{proof}
Let $ s,t\in I $, let $ x\in C(s) $ and choose a subdivision $T_0<T_{1}<...<T_{p}=T$ such that
% fix any $p_{0}\geq 1$  satisfying for every $p\geq p_{0}$, $k=1,...,p$,
 $\int\limits_{T_{k-1}}^{T_{k}}\lvert \dot{\upsilon}(\tau) \rvert\,d\tau<r$ for every
$k=1,\cdots,p$.
Fix any $k=1,...,p$ and $s,t\in [T_{k-1},T_{k}]$. Take any $i=1,...,m$ and note that
\begin{align}
g_{i}(t,x+\lvert \upsilon(t)-\upsilon(s)\rvert y )&=(g_{i}(t,x+ \lvert \upsilon(t)-\upsilon(s)\rvert y)-g_{i}(s,x+\lvert \upsilon(t)-\upsilon(s)\rvert y))\notag\\
&+g_{i}(s,x+\lvert \upsilon(t)-\upsilon(s)\rvert y)\notag\\
&\leq \lvert w(t)-w(s)  \rvert +g_{i}(s,x+ \lvert \upsilon(t)-\upsilon(s)\rvert y)\notag\\
&=\lvert w(t)-w(s)  \rvert +g_{i}(s,x)\notag\\
&+\int\limits_{0}^{1}\langle\nabla_2 g_{i}(s,x+\theta y\lvert \upsilon(t)-\upsilon(s)\rvert ),y\lvert \upsilon(t)-\upsilon(s)\rvert\rangle d\,\theta.
\end{align}
According to \eqref{50} and to the inclusion $x\in C(s)$ it ensues that
\begin{equation*}
g_{i}(t,x+\lvert \upsilon(t)-\upsilon(s)\rvert y )\leq \lvert w(t)-w(s)  \rvert - \delta  \lvert \upsilon(t)-\upsilon(s)\rvert \leq 0.
\end{equation*}
 This being true for every $i=1,...,m$, it follows that
$x+\lvert \upsilon(t)-\upsilon(s)\rvert y$ belongs to $C(t)$, otherwise stated, $x\in C(t)+\lvert \upsilon(t)-\upsilon(s)\rvert (-y)$. It results that $  C(s)\subset C(t)+\lvert \upsilon(t)-\upsilon(s)\rvert \mathbb{B} . $
Since the variables s and t play symmetric roles, the set-valued mapping $C(\cdot)$ has an absolutely continuous variation  on $[T_{k-1},T_{k}]$. From this we
clearly derive that $C(\cdot)$ has an absolutely continuous variation on $I$.
\end{proof}

\begin{example}
Let $m=1$, $n=2$, $ T=1 $,   $g(t,x)=t^{\frac{1}{3}}-x_{1}-x_{2}^{2}$,  and define
\begin{equation*}
C(t)=\{ x\in \mathbb{R}^{2} : g(t,x)\leq 0 \} .
\end{equation*}
Clearly that $C(t)$ is r-prox-regular, since $g(t,\cdot)$ satisfies all assumptions of theorem \ref{61}  for all $t\in I$. Now we check \eqref{70} and \eqref{50}. Let $x\in \mathbb{R}^{2}$, $t,s\in I$. Fix any $\delta\in (0,1]$ and put $y=(1,0)$. Then
\begin{equation*}
g(t,x)-g(s,x)=t^{\frac{1}{3}}-s^{\frac{1}{3}}\leq \lvert t^{\frac{1}{3}}-s^{\frac{1}{3}} \rvert= \lvert w(t)-w(s) \rvert,
\end{equation*}
\begin{equation*}
\langle \nabla  g(t,\cdot)(x),y \rangle=-1\leq -\delta .
\end{equation*}
We see that $w(t)= t^{1/3}$ is not Lipschitz on $ I $ but it is absolutely continuous there. Then $C(\cdot)$  has an absolutely continuous variation $\upsilon$ on $I$, with $\upsilon(t)=t^{1/3}/\delta$.
% $\dfrac{\lvert t^{\frac{1}{3}}-s^{\frac{1}{3}}  \rvert}{\delta}\leq\lvert
% \upsilon(t)-\upsilon(s)  \rvert$,  for all $t,s\in I$ .
\end{example}
\begin{theorem}
Assume that the assumptions in Theorem \ref{61}, Proposition \ref{60}   and conditions  $ \mathcal{(}\mathcal{H}_{2}\mathcal{)} $, $ \mathcal{(}\mathcal{H}_{3}\mathcal{)} $ are satisfied. Then, for every initial data $x_{0}$ with $g(0,x_{0})\leq 0$, problem (NIDCS) has one and only one solution $x(\cdot)$ .
\end{theorem}
\begin{proof}
In the same arguments like in \cite{an} one has the equivalent between the problem (NIDCS) and the integro-differential sweeping process
\begin{equation*}
-\dot{x}(t) \in N_{C(t)}(x(t))+f_{1}(t,x(t))+\int\limits_{T_0}^{t}f_{2}(t,s,x(s))\,ds
\end{equation*}
Therefore, all assumptions of Theorem \ref{exist} are satisfied and the conclusion follows.
\end{proof}

%%%%%%%%%%%%%%%%%%%%%%%%%%%%%%%%%%%%%%%%%%%%%%%%%%%%%%%%%%%%%%%%%%%%%%%%%%%%%%%%%%%%%%%%%%%%%%%%%%%%%%%%%%%%%%%%%%%%%%%%%%%%%%%%%%%%%%%%%%%%%%%%%%%%%%%%%%%%%%%%%%%%%%%%%%%%%%%%%%%%%%%%%%%%%%%%%%%%%%%%%%%%%%%%%%%%%%%%%%%%%%%%%%%%%%%%%%%%%%%%%%%%%%%%%%%%%%%%%%%%%%%%%%%%%%%%%%%%%%%%

\section{ Applications to non-regular electrical circuits}\label{s6}\ \\

The aim of this section is to illustrate the integro-differential sweeping process in the theory of non-regular
electrical circuits. Electrical devices like diodes are described in terms of Ampere-Volt
characteristic which is (possibly) a multifunction expressing the difference of
potential $v_D$ across the device as a function of current $i_{D}$ going through the device\cite{bog}.\\
Let us consider the electrical system shown in Fig. \ref{fig2} that is composed of three resistors $R_1\geq 0$, $R_2\geq 0$ with voltage/current laws $V_{R_k}=R_k x_k$ ($k=1,2$), two inductors $L_1\geq 0$, $L_2\geq 0$ with voltage/current laws $V_{L_k}=R_k \dot x_k$ ($k=1,2$), three capacitors with a time-varying capacitances $C_1 (t)\neq 0$, $C_2 (t)\neq 0$ and $C_3 (t)\neq 0$ with voltage/current laws $V_{C_k}=\frac{1}{C_k(t)} \int x_k(t)dt,\;k=1,2,3$, two ideal diodes with characteristics $0\leq -V_{D_k}\perp i_k\geq 0$ and an absolutely continuous current source $i:[0, T]\rightarrow R$.
\begin{figure}
\begin{center}
\includegraphics[width=6in]{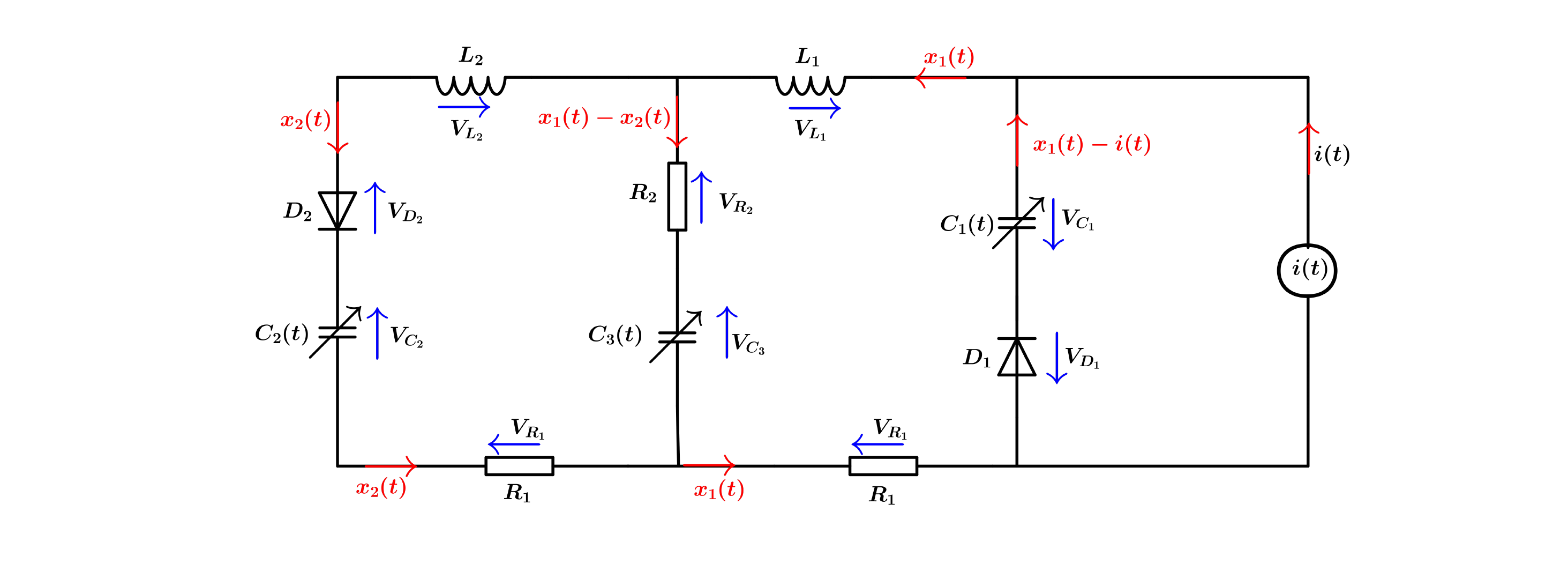}
\caption{Electrical circuit with resistors, inductances, time-varying capacitors and ideal diodes.}
\label{fig2}
\end{center}
\end{figure}

Using  Kirchhoff's laws, we have
$$\left\{
\begin{array}{l}
V_{R_1}+V_{R_2}+V_{L_1}+V_{C_1}+V_{C_3}=-V_{D_1}\in -N(\mathbb{R}_+;x_1-i)\\
V_{R_1}-V_{R_2}+V_{L_2}+V_{C_2}-V_{C_3}=-V_{D_2}\in -N(\mathbb{R}_+;x_2).
\end{array}
\right.
$$
Therefore the dynamics of this circuit is given by
\begin{align}\label{eq2.3}
\overbrace{
\begin{pmatrix}
-\dot x_1(t)\\
-\dot x_2(t)
\end{pmatrix}}^{-\dot x(t)}&\in N_{[i(t),+\infty[\times[0,+\infty[}(x(t))+\overbrace{\begin{pmatrix}
\frac{R_1 +R_2}{L_1}& -\frac{R_2}{L_1}\\
-\frac{R_2}{L_2}&\frac{R_1 +R_2}{L_2}
\end{pmatrix}}^{A_1}\overbrace{
\begin{pmatrix}
x_1 (t)\\
x_2 (t)
\end{pmatrix}}^{x(t)}\notag\\
&+\displaystyle\int\limits_{0}^{t} \bigg[ \overbrace{\begin{pmatrix}
\frac{1}{L_1 C_1(t)}+\frac{1}{L_1 C_3(t)}& -\frac{1}{L_1 C_3(t)}\\
-\frac{1}{L_2 C_3(t)} &\frac{1}{L_2 C_2(t)}+\frac{1}{L_2 C_3(t)}
\end{pmatrix}}^{A_2}\overbrace{
\begin{pmatrix}
x_1(s)\\
 x_2(s)
\end{pmatrix}}^{x(s))}
+\overbrace{
\begin{pmatrix}
\frac{1}{L_1 C_1(t)}i(s)\\
 0
\end{pmatrix}}\bigg] ds.
\end{align}

\begin{proposition}
Assume that $i : [0,T] \longrightarrow \mathbb{R}$ is an absolutely continuous function and $C_{k} : [0,T] \longrightarrow \mathbb{R^{*}},k=1,2,3$ are continuous functions. Then  for any initial condition $x(0)=x_{0}\in C(0)$, problem $\eqref{eq2.3}$ has one and only one absolutely continuous solution $x(\cdot)$.
\end{proposition}
\begin{proof}
Put $w(t)=(i(t),0)^{t}$, $C(t):=w(t)+[0,+\infty[\times [0,+\infty[$, $f_{1}(t,x)\vspace*{0.1cm}=A_1 x$, $f_{2}(t,s,x)=A_2 (t)x+\dfrac{1}{L_1 C_1(t)}w(s)$.
 So (\ref{eq2.3}) can be rewritten in the
frame of our problem $(P_{f_{1},f_{2}})$ as
\begin{equation*}
  \begin{cases} -\dot{x}(t)\in N_{C(t)}(x(t))+f_{1}(t,x(t))+\displaystyle\int\limits_{0}^{t}f_{2}(t,s, x(s))ds\ a.e.\ in \ [0, T]&\\
  x(0)=x_0\in C(0)
 \end{cases}
   \end{equation*}
Then the above data satisfying  all the assumptions of Theorem \ref{exist} ( precisely case 3 ), with
\begin{equation*}
\upsilon(t)=\int\limits_{0}^{t}\lVert \dot{w}(s) \rVert\,ds,\,\,\,\beta_{1}(t)=\lVert A_{1}\rVert,\,\,\,g(t,s)=\dfrac{1}{L_1 C_1(t)}\lVert w(s) \rVert,\,\,\,\alpha_{2}(t)=\lVert A_{2}(t)\rVert.
\end{equation*}
This finishes the proof.
\end{proof}
\bibliographystyle{amsplain}

\end{document}